\documentclass[article,hidelinks,onefignum,onetabnum]{siamart250211}

\usepackage{lipsum}
\usepackage{amsfonts}
\usepackage{graphicx}
\usepackage{epstopdf}
\usepackage{comment}
\usepackage{tikz}
\usepackage{subfig}
\usepackage{multirow}%
\usepackage[percent]{overpic}
\usepackage{algorithmic}
\usetikzlibrary{arrows.meta, positioning}
\ifpdf
  \DeclareGraphicsExtensions{.eps,.pdf,.png,.jpg}
\else
  \DeclareGraphicsExtensions{.eps}
\fi

\newsiamremark{remark}{Remark}
\newsiamremark{hypothesis}{Hypothesis}
\crefname{hypothesis}{Hypothesis}{Hypotheses}
\newsiamthm{claim}{Claim}
\newsiamremark{fact}{Fact}
\crefname{fact}{Fact}{Facts}

\headers{Stage-wise image restoration}{Liang Luo, Lei Zhang}

\title{A Semi-Convergent Stage-Wise Framework with Provable Global Convergence for Adaptive Total Variation Regularization\thanks{Submitted to the editors DATE.
\funding{This work is supported by the Zhejiang Provincial Natural Science Foundation under Grants LZ23A010006 and LY23A010004, NNSF of China under Grant 12271482, Research Grant of Zhejiang University of Technology 2022109001429.}}}

\author{Liang Luo \thanks{School of Mathematical Sciences, Zhejiang University of Technology, Hangzhou, China
  (\email{llmath09@163.com}).}
\and Lei Zhang\thanks{Corresponding author, School of Mathematical Sciences, Zhejiang University of Technology, Hangzhou, China
  (\email{zhanglei@zjut.edu.cn}) }.}
\usepackage{amsopn}

\ifpdf
\hypersetup{
  pdftitle={An Example Article},
  pdfauthor={D}
}
\fi

\begin{document}

\setlength{\textfloatsep}{5pt plus 1.0pt minus 3.0pt}
\setlength{\intextsep}{5pt plus 1.0pt minus 3.0pt}

\maketitle

\begin{abstract}
Image restoration requires a careful balance between noise suppression and structure preservation. While first-order total variation (TV) regularization effectively preserves edges, it often introduces staircase artifacts, whereas higher-order TV removes such artifacts but oversmooths fine details. To reconcile these competing effects, we propose a semi-convergent stage-wise framework that sequentially integrates first- and higher-order TV regularizers within an iterative restoration process implemented via ADMM. Each stage exhibits semi-convergence behavior—the iterates initially approach the ground truth before being degraded by over-regularization. By monitoring this evolution, the algorithm adaptively selects the locally optimal iterate (e.g., with the highest PSNR) and propagates it as the initial point for the next stage. This select-and-propagate mechanism effectively transfers local semi-convergence into a globally convergent iterative process. We establish theoretical guarantees showing that the sequence of stage-wise iterates is bounded, the objective values decrease monotonically. Extensive numerical experiments on denoising and deblurring benchmarks confirm that the proposed method achieves superior quantitative and perceptual performance compared with conventional first-, higher-order, hybrid TV methods, and 
learning based methods, while maintaining theoretical interpretability and algorithmic simplicity.
\end{abstract}



\section{Introduction}
Image restoration is a fundamental task in image processing, aiming to recover a clean image from its noisy observation while preserving critical details. In traditional approaches, whether based on linear filters or variational models, noise can be effectively reduced. However, these methods may blur fine structures or introduce artifacts, especially when dealing with high-level noise \cite{adam2021study,adam2021combined,nasonov2018improvement}. 
 To achieving a good balance between noise reduction and structure preservation remains a significant challenge. 
 
 Let ${u} \in \mathbb{R}^{MN \times 1}$ and ${f} \in \mathbb{R}^{MN \times 1}$ denote respectively the desired image and the degraded image. Then, the degraded image $f$ can be turned into the following inverse problem.
\label{noise model}
\begin{equation}
     {f} = {K}{u} + \eta,
\end{equation}
where $\eta \in \mathbb{R}^{MN \times1}$ denotes the Gaussian white noise
and  ${K} \in \mathbb{R}^{MN \times MN}$ represents the blurred operator.

Due to the reconstruction of the desired image $u$ is a severely ill-posed problem, it is not possible to obtain a unique solution \cite{bereziat2011solving,chen2023multi}. In order to address the issue, regularization techniques are employed. These techniques introduce additional constraints or prior knowledge into the model, which helps stabilize the inversion process and ensures that the solution is meaningful and well-posed \cite{afraites2022weighted,hintermuller2017optimal}. By penalizing undesirable features, regularization guides the solution toward a more robust and realistic result.  For image restoration, we usually describe it using the following optimization problem.
\begin{equation}
\label{optimization problem}
 u=\arg \min\limits_{u} \mathcal{L} (Ku,f) + \lambda\,  \phi({u})
\end{equation}
where $\mathcal{L}(Ku,f)$ denotes data fidelity, $\lambda>0$ represents the regularization parameter, and $\phi({u})$ denotes the regularization function. 

In problem \cref{optimization problem}, the similarity between the observed image and the estimated image is guaranteed by the data fidelity term.The regularization term encodes prior knowledge and helps to extract important structural information from the image. Among various regularization techniques, the TV regularization, originally proposed by Rudin et al. \cite{rudin1992nonlinear}, has become one of the most influential methods in image processing. It is formulated as follows.
\begin{equation}
\label{ROF model}
       u =  \arg \min\limits_{u}  \frac{1}{2} \| {K}{u}-{f}\|_2^2 + \lambda \, \phi({Du})
\end{equation}
where $\phi({Du}) = \sum\limits_{i\leq M,j\leq N}|({Du})_{i,j}|$ and ${K}$ denotes a convolution i.e., a blurring matrix. For image denoising task, $K$ is the identity operator.

By using first-order TV regularization, some satisfactory results have been obtained in image denoising and edge preservation\cite{guo2021image,khan2017mesh}.
However, such a model tends to produce blocky artifacts in uniform regions, which will significantly degrade image quality \cite{papafitsoros2015novel,xu2014coupled,yin2022ℓ0}.To address this drawback, several methods have been developed and have shown promising results in suppressing such artifacts while achieving better image restoration quality \cite{liu2015image,adam2019image,beck2016rate,bredies2010total}. Among these, the most fundamental higher-order TV model was introduced by Lysaker et al. \cite{lysaker2003noise}, and it can be written as
\begin{equation}
\label{LLT model}
     u = \arg \min\limits_{{u}} \frac{1}{2} \| {K}{u}-{f}\|_2^2 + \lambda \,  \phi({DDu}) 
\end{equation}
where $\phi({DDu}) = \sum\limits_{i\leq M, j\leq N}|({DDu})_{i,j}|$. 

The higher-order TV model (\ref{LLT model}) exceeds the first-order TV model (\ref{ROF model}), the limitations of higher-order TV regularization have been indicated, showing that these methods may degrade edge sharpness and lead to the loss of fine textural details \cite{papafitsoros2014combined}.

Building upon the aforementioned models, an iterative image restoration (IIR) method \cite{lysaker2006iterative} was proposed, utilizing initially a weighting function $\theta$ to combine the outputs of the first- and higher- order TV models. This hybrid approach leverages the complementary strengths of both models to achieve a trade-off between noise reduction and artifact suppression. Motivated by this idea, parallel hybrid approaches that combine the first-order TV term with higher-order TV terms and overlapping group sparsity (OGS) were proposed in \cite{li2007image,pang2011proximal,oh2013non,jon2021image,papafitsoros2014combined,bhutto2023image}. These methods can reduce noise, avoid artifacts in smooth regions, and achieve better restoration results.

Although significant progress has been made in optimizing regularization terms in existing research, the semi-convergence problem of iterative methods in image denoising cannot be effectively avoided. In this work, we propose a novel stage-wise image denoising framework that combines the strengths of both first-order and higher-order TV regularizations. In each stage of the framework, these two regularization models are applied sequentially within an ADMM optimization scheme. The design is motivated by the complementary characteristics of first- and higher-order TV: the former excels at preserving sharp edges, while the latter suppresses staircase artifacts and residual noise.

A key feature of the proposed method lies in its semi-convergent refinement strategy. Specifically, we exploit the semi-convergence properties of iterative regularization by monitoring the PSNR evolution within each stage: the intermediate result with the higher PSNR is regarded as a locally optimal solution, which serves as the initial guess for the subsequent stage. This strategy allows the overall process to adaptively refine the reconstruction and improves stability. We further establish the convergence of this stage-wise process both theoretically and numerically, and extensive experiments on standard benchmarks show that our method consistently outperforms single-model TV-based baselines in both quantitative metrics (e.g., PSNR, SSIM, FOM) and visual quality.

The remainder of this paper is structured as follows. Section~\ref{sec2} reviews the necessary preliminaries and the ADMM framework. Section \ref{sec3} discusses the optimization procedure for the proposed method and provides a theoretical convergence analysis. Section~\ref{sec4} presents numerical experiments to validate the convergence behavior and restoration performance of the proposed framework. Finally, Section~\ref{sec5} summarizes the work and presents directions for future research.

\section{Preliminary}
\label{sec2}
 A digital image can be modeled as a two-dimensional function $u(i,j)$, where $i$ and $j$ denote respectively the spatial coordinates.
For the image of size $n = M \times N$, its discrete representation will form  a matrix.
In typical practice, the image is vectorized into an $n$-dimensional column vector. In the following   subsection, we introduce some basic notions and review ADMM algorithms.

\subsection{The Notion of TV Operators}
The discrete first-order and second-order gradient operators, denoted respectively as $D$ and $DD$, are constructed following the formulation presented in Wu and Tai~\cite{wu2010augmented}. Accordingly, the first-order discrete gradient $ ({Du})_{i,j} $ is given by 

 \begin{equation}
    ({Du})_{i,j} = (({D}_x^+{u})_{i,j},({D}_y^+{u})_{i,j})\\
 \end{equation}
where ${D}_x^+{u}$ and ${D}_y^+{u}$ are the first-order forward differences 
defined respectively as
\begin{equation}
(D_x^+u)_{i,j} = \begin{cases}
u_{i,j+1} - u_{i,j}, & 1 \leq j \leq N-1 \\
u_{i,1} - u_{i,N}, & j = N
\end{cases}
\end{equation}
\begin{equation}
(D_y^+u)_{i,j} = \begin{cases}
u_{i+1,j} - u_{i,j}, & 1 \leq i \leq M-1 \\
u_{1,j} - u_{N,j}, & i = M
\end{cases}    
\end{equation}
and the first-order backward differences ${D}_x^-{u}$ and ${D}_y^-{u}$ separately are
\begin{equation}
(D_x^-u)_{i,j} = \begin{cases}
u_{i,j} - u_{i,j-1}, & 1 < j \leq N \\
u_{i,1} - u_{i,N}, & j = 1
\end{cases}
\end{equation}
\begin{equation}
(D_y^-u)_{i,j} = \begin{cases}
u_{i,j} - u_{i-1,j}, & 1 \leq i \leq N \\
u_{1,j} - u_{n,j}, & i = 1
\end{cases} 
\end{equation}
and the second-order discrete gradient $({DDu})_{i,j}$ can be described as
\begin{equation}
(DDu)_{i,j} =
\begin{pmatrix}
{D}_x^-(D_x^+u_{i,j}) & {D_y^+}(D_x^+u_{i,j}) \\ 
{D_x^+}(D_y^+ u_{i,j}) & {D_y^-}(D_y^+ u_{i,j})
\end{pmatrix} 
\end{equation}

Detailed definitions of the discrete second-order gradient operators $D_x^-(D_x^+)$, $D_y^+(D_x^+)$, $D_x^+(D_y^+)$, and $D_y^-(D_y^+)$ can be found in ~\cite{wu2010augmented}.

\subsection{ADMM algorithm}
As an enhanced variant of the classical augmented Lagrangian approach \cite{de2023constrained}, the alternating direction method of multipliers (ADMM) \cite{han2012note,xiang2023poisson} decomposes multi-variable optimization tasks into smaller sub-problems for iterative solutions. ADMM has been shown to be highly effective for constrained convex optimization problems with separable structures, due to this decomposition framework.

\begin{equation}
\label{ADMM}
\min\limits_{{x,z}}f(x) + g(z),\ \text{s.t.} \ Ax + Bz = c.
\end{equation}
where $f(\cdot)$, $g(\cdot)$ are closed convex functions,  $A \in \mathbb{R}^{p \times n} $, $B \in \mathbb{R}^{p \times m} $ are linear operator, $x \in \mathbb{R}^n$, $z \in \mathbb{R}^m$ are elements of nonempty closed convex sets, $c \in \mathbb{R}^p$ represents a given vector. The augmented Lagrangian function corresponding to problem \cref{ADMM} takes the form
\begin{equation}
\label{ALF}
    L_{\rho}(x,z,\mu) =f(x) + g(z) - {\mu}^T(Ax+Bz-c) 
    + \frac{\rho}{2} \| Ax+Bz-c \|_2^2 
\end{equation}
where $\mu \in \mathbb{R}^p$ denotes the Lagrange multiplier and $\rho >0$ is the penalty parameter. 

The saddle point of problem \cref{ALF} is achieved by alternatively minimizing $L_{\rho}$ over $x$, $z$, and $\mu$. The ADMM algorithm for optimizing problem \cref{ADMM} is outlined in \cref{algo1}.

\begin{algorithm}
\caption{ADMM for the optimization problem \cref{ADMM}}
\label{algo1}
\begin{algorithmic}
\STATE{Initialize $x^0$, $z^0$, $\mu^0$ and $\rho >0$.}
\WHILE{stopping condition is not satisfied}
\STATE{$x^{k+1} = \arg\min\limits_{x} \left( f(x) + \frac{\rho}{2} \| Ax + Bz - c + \frac{\mu^k}{\rho} \|_2^2 \right)$} \text{ // Update $x$}
\STATE{$z^{k+1} = \arg\min\limits_{z} \left( g(z) + \frac{\rho}{2} \| Ax^{k+1} + Bz - c + \frac{\mu^k}{\rho} \|_2^2 \right)$} \text{//Update $z$}
\STATE {$\mu^{k+1} = \mu^k + \rho (Ax^{k+1} + Bz^{k+1} - c)$} \text{// Update $\mu$}
\STATE{$k = k + 1$} 
\ENDWHILE
\RETURN $x^{k+1}$
\end{algorithmic}
\end{algorithm}

\section{Related work}
\label{sec3}
To formulate the restoration framework, we propose a hybrid regularization model 
that alternates between first-order and higher-order TV terms. 
For convenience, we introduce a switching function $\sigma(i)$. 
The objective function takes the following form:
\begin{equation}
\label{Coupled problem}
\mathcal{J}(u)
= \frac{1}{2}\,\|Ku-f\|_2^2
+ \lambda_{\sigma(i)}\, \mathrm{TV}_{\sigma(i)}(u),
\qquad i=1,2\dots
\end{equation}
where
\[
\sigma(i)=
\begin{cases}
1, & \text{$i$ is odd}\\
2, & \text{$i$ is even}
\end{cases}
\]
which selects the type of TV regularizer used in the $i$-th stage. 
Specifically, $\mathrm{TV}_1(u)=\phi(Du)$ and $\mathrm{TV}_2(u)=\phi(DDu)$ 
denote respectively the first- and higher- order TV regularizer with $\lambda_1$ and $\lambda_2$ being their corresponding regularization weights. 
Here, $K$ denotes the blurring operator and $f$ is the degraded observation. 
The switching function $\sigma(i)$ ensures that the two regularizers 
are not minimized simultaneously but are instead updated sequentially in an alternating manner.

As the first step in our alternating optimization strategy, we first minimize the first-order TV sub-problem associated with the term  $\lambda_1 \mathrm{TV}_1(u)$. This is equivalent to solving the classical ROF model via the ADMM algorithm described in ~\cite{wu2010augmented}, and the details are given in \cref{algo2}.

\begin{algorithm}
\caption{ADMM for the optimization problem \cref{ROF model}}
\label{algo2}
\begin{algorithmic}

\STATE{Require $f$, $K$ and parameters $\lambda>0$, $\omega>0$.}
\STATE{Initialize $u_0=f$, $k=0$, $\rho=2$, $y_i = 0$ for $i=1,2$.}
\WHILE{stopping condition is not satisfied}
\STATE{${u}^{k+1} = \mathcal{F}^{-1}(\frac{ \mathcal{F}{(K^T)}\mathcal{F}({f})- \mathcal{F}{(D^T)}\mathcal{F}(\rho v^k + y^k )}{\mathcal{F}({K^T})\mathcal{F}({K}) + \rho\mathcal{F}(D^T)\mathcal{F}(D)})$} \text{ // Update $x$}
\STATE{$v^{k+1} = \max \left\{0, \|x^{k+1}\|_2 - \frac{\lambda}{\rho} \right\} \cdot sign(x^{k+1})$} \text{//Update $v$}
\STATE {$y^{k+1} = y^k + \rho (v^{k+1} - Du^{k+1})$} \text{// Update $y$}
\STATE{$k = k + 1$} 
\ENDWHILE
\RETURN $x^{k+1}$
\end{algorithmic}
\end{algorithm}

The restored result from the first-order TV sub-problem is then passed to the second stage, where we solve the higher-order TV sub-problem corresponding to the term  \( \lambda_2 \mathrm{TV}_2(u) \). This stage aims to further suppress staircase artifacts and enhance smoothness in homogeneous regions. To achieve this, we formulate the second-stage problem as the minimization of the following energy functional:

\begin{equation}
\label{Hotv}
\begin{split}
 \min \limits_u \frac{1}{2} \| Ku - f \|_2^2 + \lambda \cdot \phi(DDu)
\end{split}
\end{equation}
where $\lambda>0$ is the regularization parameter.

To facilitate solving problem \cref{Hotv}, we define an intermediate variable ${v}$ and recast the problem into a constrained formulation.
\begin{equation}
\label{Constrained problem}
\begin{split}
\min\limits_{{u,v}} \frac{1}{2} \| {Ku - f} \|_2^2 + \lambda \cdot \phi({v}), \ \text{s.t.} \ {v} = {DDu}.
\end{split}
\end{equation}

The Lagrange multipliers ${\mu}$ and the penalty parameter $\beta$ are introduced. Therefore, the augmented Lagrangian function of problem \cref{Constrained problem} is
\begin{equation}
\label{augmented Lagrangian function}
    \begin{split}
    L_{\beta}({u},{v},{\mu}) = \frac{1}{2} \, \|{Ku-f}\|_2^2 + \lambda \cdot \phi({v}) - {\mu}^T({v-DD}{u})
    +\frac{\beta}{2} \, \|{v-DDu}\|_2^2
    \end{split}  
\end{equation}

In problem \cref{augmented Lagrangian function}, the intermediate variables ${v}$ and Lagrange multipliers ${\mu}$ can  be described as 
\begin{align}
    \begin{aligned}
     {v} = [{v_1},{v_2},{v_3},{v_4}]^T,\ \  {\mu}= [{\mu_1},{\mu_2},{\mu_3},{\mu_4}]^T
    \end{aligned}
\end{align}
respectively, and
\begin{equation}
    \|{v}\|_2 = \sum\limits_{i\leq M,j\leq N} \sqrt{{(v_1)}_{i,j}^2 + {(v_2)}_{i,j}^2 + {(v_3)}_{i,j}^2 + {(v_4)}_{i,j}^2}
\end{equation}
 
To obtain the saddle point of problem \cref{augmented Lagrangian function},
 we use the ADMM method to iteratively and alternately solve the following  three sub-problems
\begin{align}
    \left\{
    \begin{aligned}
    u^{k+1} =& \arg \min \limits_u \frac{1}{2} \| Ku - f \|_2^2  - \langle \mu^k,v^k-DDu \rangle +\frac{\beta}{2} \| v^k - DDu \|_2^2,  \\
    v^{k+1} =& \arg \min\limits_v \lambda \cdot  \, \phi(v) - \langle \mu^k,v-DDu^{k+1} \rangle+ \frac{\beta}{2} \| v - DDu^{k+1} \|_2^2, \\
    \mu^{k+1} =& \mu^k - \beta (v^{k+1} - DDu^{k+1}).
    \end{aligned}
    \right.
\end{align}
The following steps provide a detailed solution process for the given problems.

\subsection{u sub-problem} fixing  $v$ and $\mu$, the sub-problem for $u^{k+1}$ can be rewritten as:
\begin{equation}
\label{u-subproblem}
\begin{split}
    u^{k+1} =& \arg \min\limits_{u}  \frac{1}{2} \| Ku - f \|_2^2  
    + \frac{\beta}{2} \left\| v^k - DDu - \frac{\mu^k}{\beta} \right\|_2^2
\end{split}
\end{equation}

By solving the sub-problem \cref{u-subproblem}, we have
\begin{equation}
\label{linear system}
\begin{split}
    ( K^T&K + \beta (DD)^T DD) u^{k+1} 
    =  K^T f +\beta (DD)^T v^k - (DD)^T \mu^k
\end{split}
\end{equation}

Since periodic boundary conditions are applied, and the matrices $K^TK$ and $(DD)^TDD$ are block circulant \cite{wang2008new}, the linear system \cref{linear system} can be efficiently solved using the Fast Fourier Transform (FFT), and the solution for $u^{k+1}$ is obtained via element-wise division.

\begin{equation}
{u}^{k+1} =  \mathcal{F}^{-1} \left( \frac{ \mathcal{F}(K^T) \mathcal{F}(f) + \mathcal{F}(DD)^T \mathcal{F}(\beta v^k - \mu^k)}{ \mathcal{F}(K^T) \mathcal{F}(K) + \beta \mathcal{F}(DD)^T\mathcal{F}(DD)} \right)
\end{equation}
where $\mathcal{F}$ and $\mathcal{F}^{-1}$ respectively represent FFT and its inverse transform.

\subsection{v sub-problem}fixing variable $u$ and $\mu$. The $v$-sub-problem can be rewritten as 
\begin{equation}
\label{v-subproblem}
\begin{split}
    v^{k+1} =& \arg \min\limits_{v}  \lambda \cdot \phi(v) 
    +\frac{\beta}{2} \left\| v - (DDu^{k+1} + \frac{\mu^k}{\beta}) \right\|_2^2
\end{split}
\end{equation}

Let $x^{k+1} = DDu^{k+1} + \frac{\mu^k}{\beta}$, the solution of Problem \cref{v-subproblem} can be obtained using the well-known shrinkage formula. 
\begin{equation}
    \begin{split}
        v^{k+1} =  \mathrm{shrink} \, (x^{k+1}, \frac{\lambda}{\beta})
        =\max \, \{0,\|{x}^{k+1}\|_2 - \frac{\lambda}{\beta}\} \cdot sign(x^{k+1})  \\
    \end{split}
\end{equation}
where $\mathrm{shrink} \,(\cdot)$ denotes the shrinkage operator and $\|x^{k+1}\|_2$ is given by
\begin{equation}
\|x^{k+1}\|_2 = \sum\limits_{i=1}^{M} \sum\limits_{j=1}^{N} 
\sqrt{(x_1^{k+1})_{i,j}^2 + (x_2^{k+1})_{i,j}^2 + (x_3^{k+1})_{i,j}^2 + (x_4^{k+1})_{i,j}^2}
\end{equation}

Hence, the detailed process for the higher-order TV model \cref{LLT model} is presented in ~\cref{algo3}. 

\begin{algorithm}
\caption{ADMM for the optimization problem \cref{LLT model}}
\label{algo3}
\begin{algorithmic}

\STATE{Require $f$, $K$ and parameters $\lambda > 0$, $\omega>0$.}
\STATE{Initialize $u_0 = f$, $k = 0$, $\beta =2$ and $\mu_j = 0$ for $j = 1, 2, 3, 4$.}
\WHILE{stopping condition is not satisfied}
\STATE{${u}^{k+1} = \mathcal{F}^{-1}\left( \frac{ \mathcal{F}(K^T)\mathcal{F}(f) + \mathcal{F}(D D^T)\mathcal{F}(\beta v^k - \mu^k)}{ \mathcal{F}(K^T)\mathcal{F}(K) - \beta \mathcal{F}(D D^T)\mathcal{F}(D D^T)} \right)$} \text{ // Update $u$}
\STATE{$v^{k+1} = \max \left\{0, \|x^{k+1}\|_2 - \frac{\omega}{\beta} \right\} \cdot \text{sign}(x^{k+1})$} \text{//Update $v$}
\STATE {$\mu^{k+1} = \mu^k - \beta (v^{k+1} - D D u^{k+1})$} \text{// Update $\mu$}
\STATE{$k = k + 1$} 
\ENDWHILE
\RETURN $u^{k+1}$
\end{algorithmic}
\end{algorithm}

Building upon this, we propose a stage-wise alternating restoration algorithm, where each stage solves a variational problem with either first-order or high-order total variation regularization. The alternation is governed by the principle of semi-convergence: once the residual stops decreasing and starts increasing, the iteration halts and transitions to the next stage. This strategy avoids overfitting noise and enables adaptivity to different image features. The full method alternates between the two stages until a total iteration cap is reached or the accuracy setting is achieved. The stage-wise alternating method proposed is succinctly summarized in \cref{algo4}.

\begin{algorithm}
\caption{Stage-wise alternating framework for image restoration}
\label{algo4}
\begin{algorithmic}
\STATE{Require Noisy image $f$, blurring operator $K$, regularization parameters $\lambda_1, \lambda_2$, maximum outer iterations $N_{\max}$.}
\STATE{Ensure Restored image $u$.}
\STATE{Initialize $u_0^{(k)} = f$, $k=0$, $N = 1$.}
\WHILE{$N \le N_{\max}$}
\STATE{$u_N^{(k+1)} = \textbf{Algorithm 2}(u_{N-1}^{(k)}, K, \lambda_1)$} \text{// First-order TV regularization stage }
  \IF{$\|Ku_N^{(k+1)} - f\| > \|Ku_N^{(k)} - f\|$}
    \STATE{Return $u_N^{(k)}$}  \text{// semi-convergence analysis}
  \ENDIF
  \IF{$N \geq N_{\max}$}
  \STATE{Return $u_N^{(k)}$} \text{// Check if $N_{max}$ is reached}
  \ENDIF
\STATE{$N \gets N+1$} \text{// $N+1$}
\STATE{$u_N^{(k+1)} = \textbf{Algorithm 3}(u_{N-1}^{(k)}, K, \lambda_2)$} \text{// High-order TV regularization stage}
  \IF{$\|Ku_N^{(k+1)} - f\| > \|Ku_N^{(k)} - f\|$}
    \STATE{Return $u_N^{(k)}$}  \text{// semi-convergence analysis}
  \ENDIF
  \IF{$N \geq N_{\max}$}
    \STATE{Return $u_N^{(k)}$} \text{// Check if $N_{max}$ is  reached}
  \ENDIF
    \STATE{$N \gets N+1$}  \text{// $N+1$}
\ENDWHILE
\RETURN $u_N^{(k)}$
\end{algorithmic}
\end{algorithm}

To facilitate the convergence analysis of the proposed algorithm, we introduce following notations.

we define $\mathcal{J}_1(u) = \lambda_1TV_1(u),\mathcal{J}_2(u)=\lambda_2TV_2(u)$ and then the optimization problem \cref{Coupled problem} can be divided into independent two part
\begin{equation}
\label{J(u)}
    \min_u \;  \mathcal{J}_i(u) + H(u),  \  \ {i} = {1,2}
\end{equation}
where $\mathcal{J}_i$ is proper, convex and lower semi-continuous, $H$ is convex and continuously differentiable. At the same time, the approximation of  problem  is given as follows
\begin{equation}
\label{2}
    \min_u \mathcal{J}(u) + \frac{\delta_k}{2}\| u - (u^k-\frac{1}{\delta_k}K^T(Ku^k-f))\|^2_2
\end{equation}

Building upon prior works in \cite{bredies2008linear,hong2017linear,xie2014admm}, we introduce the following lemmas.

\begin{lemma}
\label{Lemma1}
Assume that H is Lipschitz continuous with associated Lipschitz constant L, $u^*$ is the solution of problem \cref{J(u)}, $u^{k}$ is the solution of the problem \cref{2}. If, for each $M \in R$ there exists a constant $\sigma(M)>0$ such that
    \begin{equation}
        \| \bar{v}-u^*\| \le \sigma(M)(R(\bar{v}+T(\bar{v})))
    \end{equation}
    for each $\bar{v}$ satisfying $(\mathcal{J}+H)(u)\le M$ and $\bar{v}$,$R(\bar{v})$ and $T(\bar{v})$ defined by 
    \begin{align}
        \bar{v} = \arg\min_v \mathcal{J}(v) + \frac{\delta}{2} \|v-u+\frac{1}{\delta}H'(u)\|^2_2 \\
        R(\bar{v}) = \langle H'(u^*),\bar{v}-u^*\rangle + \mathcal{J}(\bar{v}-\mathcal{J}(u^*)),\\
        T(\bar{v}) = H(\bar{v})-H(u^*)-\langle H'(u^*),\bar{v}-u^* \rangle
    \end{align}
    then ${u^{k}}$ converges linearly to the unique minimizer $u^*$.
\end{lemma}

\begin{lemma}
\label{Lemma2}
The sequence of iterates $\{u^k\}$ generated by the ADMM Algorithm \cref{algo2} converges linearly to the solution of the  problem \cref{ROF model}.
\end{lemma}

\begin{lemma}
\label{Lemma3}
The sequence of iterates $\{u^k\}$ generated by the ADMM Algorithm \cref{algo3} converges linearly to the solution of the problem \cref{LLT model}.
\end{lemma}

For the purpose of analyzing the convergence of the proposed method, 
we assume the existence of a neighborhood $B(u,\delta)$ of the ground truth image $u$ 
such that both optimal solutions $u^*_1$ and $u^*_2$ are contained in $B(u,\delta)$. 

\begin{theorem}
\label{Theorem}
The sequence of iterates $\{u^k_N\}$ converges to the neighborhood of ground image $u$, where the $u^k_N$ is generated by the \cref{algo4}. 
\end{theorem}

\begin{proof}[Proof]
Define that $\{u_N^k\}$ denotes the sequence generated by \cref{algo4} 
at the $k$-th iteration of the $N$-th stage. 
Let $u^*_1$ and $u^*_2$ denote respectively the solutions of the minimization problem \cref{ROF model} and problem \cref{LLT model}. 
Moreover, let $u$ denote the ground truth image. 
By \cref{Lemma2}, \cref{Lemma3} and the above assumptions, we proceed to prove the theorem as follows.

    \begin{align}
    \|u^k_N - u \| = & \| u^k_N - u^*_1 + u_1^* - u^*_2 + u^*_2 -u \| \\
    \le& \|u^k_N - u_1^*\| + \| u^*_1 - u^*_2\| + \|u^*_2 - u\| \\
    \le& \|u_N^k - u^*_1\| + 2\delta + \delta  \\
    \le& \sigma +3\delta \quad  \text{where}\quad N =2n-1,n=1,2,3,\cdots\\
    \|u^k_N - u \| = & \| u^k_N - u^*_2 + u_2^* - u^*_1 + u^*_1 -u \| \\
    \le& \|u^k_N - u_2^*\| + \| u^*_2 - u^*_1\| + \|u^*_1 - u\| \\
    \le& \|u_N^k - u^*_2\| + 2\delta + \delta  \\
    \le& \sigma + 3\delta \quad  \text{where}\quad N =2n,n=1,2,3,\cdots
\end{align}

According to the above discuss. we let $\hat{\delta} =\sigma +  3\delta$, then we have 
\begin{equation}
    \|u^k_N - u\| \le \hat{\delta} \quad\text{for all}\quad N
\end{equation}

Therefore, the sequence $\{u_N^k\}$ generated by \cref{algo4} converges to the neighborhood of the ground image $u$. To better understand the convergence process of the alternating iteration, we have drawn the flow diagram of the graphical convergence process, as shown in \cref{fig:tikz}.
\end{proof}

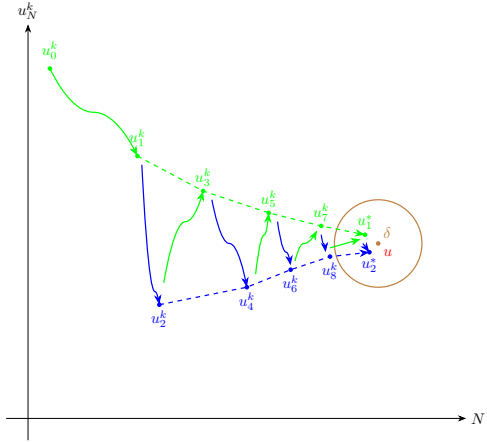
\begin{figure}[htbp]
  \centering
  \resizebox{0.5\textwidth}{!}{\begin{tikzpicture}[>=Stealth,scale=1]
\draw[->] (-0.5,0) -- (10,0) node[right] {$N$ };
\draw[->] (0,-0.5) -- (0,9) node[above] {$u_N^k$};

\draw[brown, thick] (8,4) circle (1);        
\fill[brown] (8,4) circle (1.5pt);           
\node[red] at (8,4) [below right] {$u$};  
\node[brown] at (8,4) [above right] {$\delta$};

\fill[green] (2.5,6) circle (1.5pt) node[above,yshift=3pt] {$u_1^k$};
\fill[blue] (3,2.6) circle (1.5pt) node[below] {$u_2^k$};

\fill[green] (4,5.2) circle (1.5pt) node[above] {$u_3^k$};
\fill[blue] (5,3) circle (1.5pt) node[below] {$u_4^k$};

\fill[green] (5.5,4.7) circle (1.5pt) node[above] {$u_5^k$};
\fill[blue] (6,3.4) circle (1.5pt) node[below] {$u_6^k$};

\fill[green] (6.7,4.4) circle (1.5pt) node[above] {$u_7^k$};
\fill[blue] (6.9,3.7) circle (1.5pt) node[below] {$u_8^k$};

\fill[green] (7.7,4.2) circle (1.5pt) node[above] {$u_1^*$};
\fill[blue] (7.8,3.8) circle (1.5pt) node[below] {$u_2^*$};

\draw[thick,green,->] (0.5,8) parabola bend (1.5,7) (2.5,6);
\draw[thick,blue,->] (2.6,5.8) parabola bend (2.9,2.9) (3,2.6);

\draw[thick,green,->] (3.1,3.1) parabola bend (3.6,4.5) (4,5.2);
\draw[thick,blue,->] (4.2,5.0) parabola bend (4.6,4.0) (5,3);

\draw[thick,green,->] (5.2,3.3) parabola bend (5.4,4.2) (5.5,4.7);
\draw[thick,blue,->] (5.7,4.5) parabola bend (5.9,3.8) (6,3.5);

\draw[thick,green,->] (6.1,3.6) parabola bend (6.3,4) (6.6,4.3);
\draw[thick,blue,->] (6.7,4.2) parabola bend (6.75,4) (6.8,3.8);

\draw[thick,green,->] (6.9,3.9) -- (7.6,4.1);
\draw[thick,blue,->] (7.65,4) -- (7.8,3.8);

\draw[->,green,dashed,thick] (2.5,6) -- (4,5.2) -- (5.5,4.7) -- (6.7,4.4)  -- (7.7,4.2);
\draw[->,blue,dashed,thick] (3,2.6) -- (5,3) -- (6,3.4) -- (6.9,3.7)  -- (7.8,3.8);

\fill[green] (0.5,8) circle (1.5pt) node[above,yshift=3pt] {$u_0^k$};

\end{tikzpicture}}
  \caption{The flow diagram  of the graphical convergence process}
  \label{fig:tikz}
\end{figure}

 Although the denoising framework is guided by the semi-convergence property to select intermediate iterates within each stage, \cref{Theorem} ensures that the overall iterative sequence ${u_N^k}$ converges to a neighborhood of the ground image $u$. This result is particularly significant, as traditional TV-based denoising methods often lack such global convergence guarantees and rely solely on heuristic stopping rules. The proof, while concise, confirms that the proposed scheme remains stable under the stage-wise regularization strategy.

\section{Numerical results}
\label{sec4}
In order to assess the performance of our proposed approach, we carried out numerical experiments on the test images illustrated in Figure~\ref{tested image}, considering three noise levels: $\sigma = 10$, $\sigma = 20$, and $\sigma = 30$.
Furthermore, a $3\times 3$ Gaussian kernel was applied to further blur the image. All experiments were performed on an Apple M1 with 8 GB of RAM. Practically, in this study, MATLAB (R2021b) is chosen to run the proposed procedure. PSNR \cite{mozhaeva2021full}, SSIM \cite{bakurov2022structural}, and FOM \cite{zhang2017noise} are the main quantitative analyses that are used to evaluate the proposed method. 

\begin{figure}[htbp]
\centering
\subfloat[]{\includegraphics[width=0.22\textwidth]{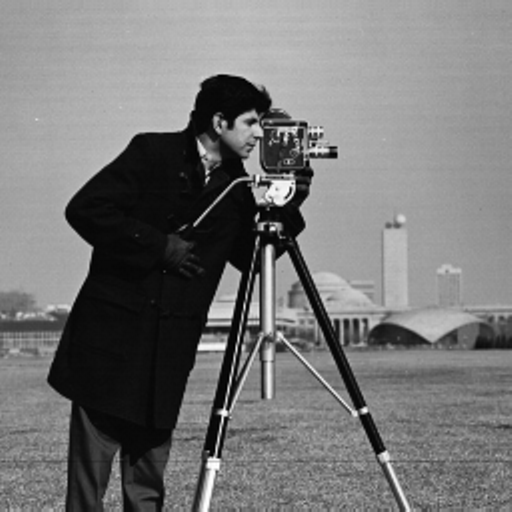}} \hspace{0.1cm}
\subfloat[]{\includegraphics[width=0.22\textwidth]{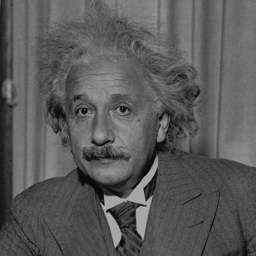}} \hspace{0.1cm}
\subfloat[]{\includegraphics[width=0.22\textwidth]{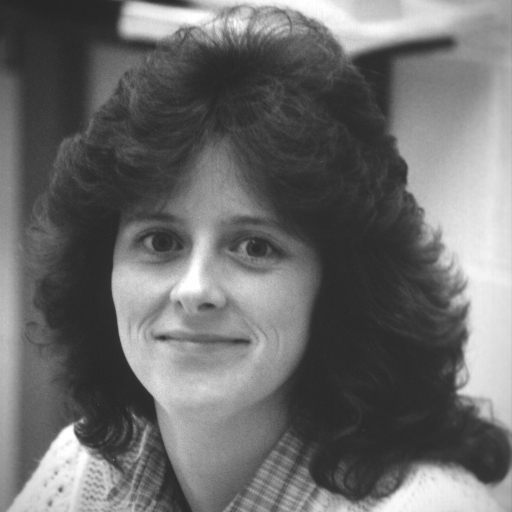}} \hspace{0.1cm}
\subfloat[]{\includegraphics[width=0.22\textwidth]{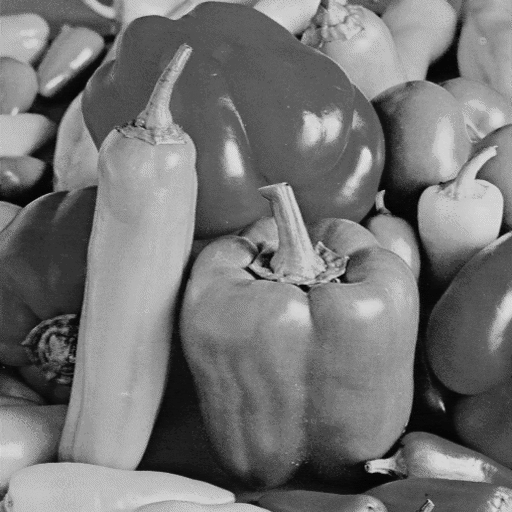}} \hspace{0.1cm}\\
\subfloat[]{\includegraphics[width=0.22\textwidth]{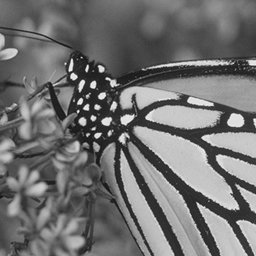}} \hspace{0.1cm}
\subfloat[]{\includegraphics[width=0.22\textwidth]{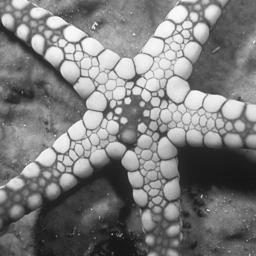}} \hspace{0.1cm}
\subfloat[]{\includegraphics[width=0.22\textwidth]{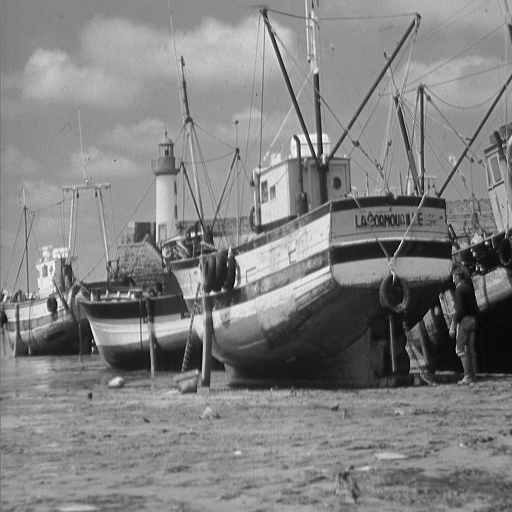}} \hspace{0.1cm}
\subfloat[]{\includegraphics[width=0.22\textwidth]{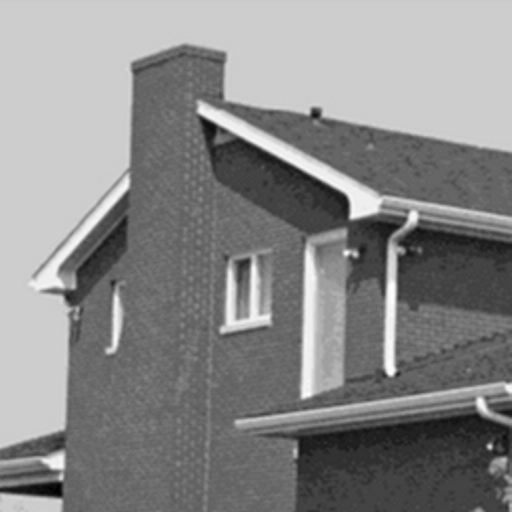}} \hspace{0.1cm}\\
\subfloat[]{\includegraphics[width=0.22\textwidth]{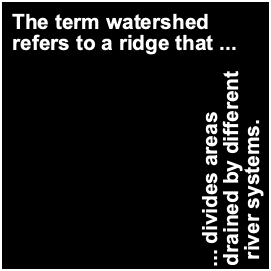}} \hspace{0.1cm}
\subfloat[]{\includegraphics[width=0.22\textwidth]{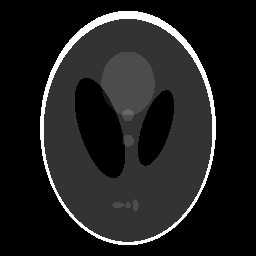}} \hspace{0.1cm}
\subfloat[]{\includegraphics[width=0.22\textwidth]{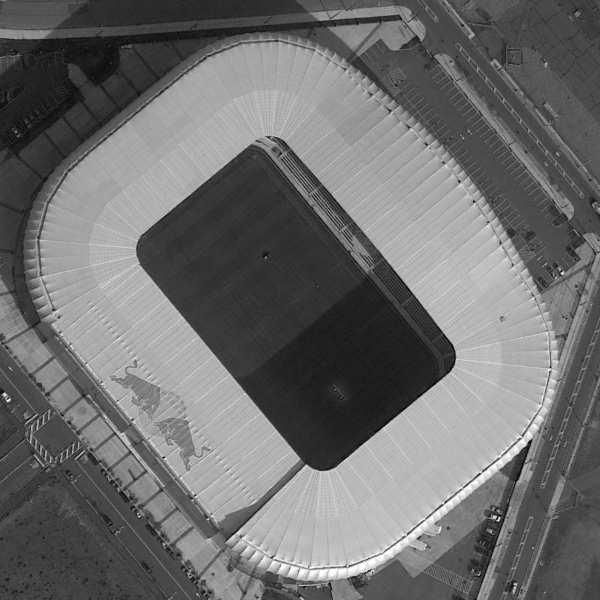}} \hspace{0.1cm}
\subfloat[]{\includegraphics[width=0.22\textwidth]{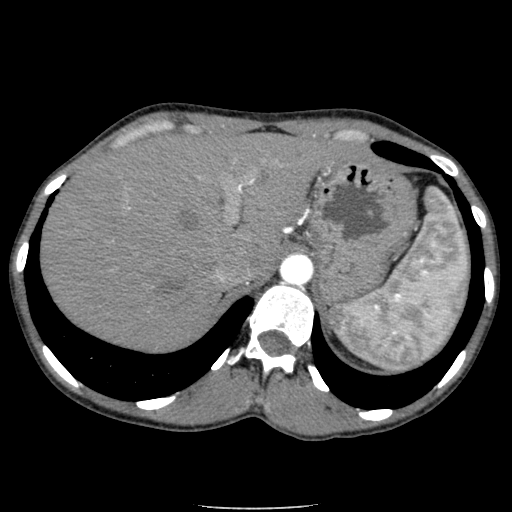}} \\[-0.2cm]

\caption{The test images. (a) Cameraman  (b) Einstein (c) Woman (d) Peppers (e) Butterfly (f) Starfish (g) Boat (h) House (i) Text (j) Logan (k) Stadium (l) Brain}
\label{tested image}
\end{figure}

We compared the proposed method with several existing TV models and methods, including the ROF model \cite{rudin1992nonlinear}, the LLT model \cite{lysaker2003noise}, the TGV model \cite{bredies2010total},  the DBC-TV model \cite{beck2016rate} and the Nesterov's method \cite{mukherjee2016study}. In the original papers \cite{rudin1992nonlinear} and \cite{lysaker2003noise}, the gradient descent scheme was used to solve the minimization problem. However, in our experiment, we employ the ADMM scheme \cite{han2012note} to improve the efficiency of the optimization process.   
Unlike standard gradient descent schemes, the ADMM framework reformulates the constrained optimization problem through decomposition into smaller sub-problems. These sub-problems are solved independently, leading to accelerated convergence and improved effectiveness for image restoration.

In all Algorithm, the stopping criterion was choosen as
\begin{equation}
     e_r = \frac{\|{u}^{k+1}-u^k\|_2}{\| u^k \|_2} <\delta
 \end{equation}
where $\delta = 1\times10^{-8}$, ${u}^{k+1}$ and ${u}^k$ correspond to the restored image obtained in the present and the previous iterations, respectively.

\begin{figure}[htbp]
\centering
\subfloat[]{\includegraphics[width=0.8\textwidth]{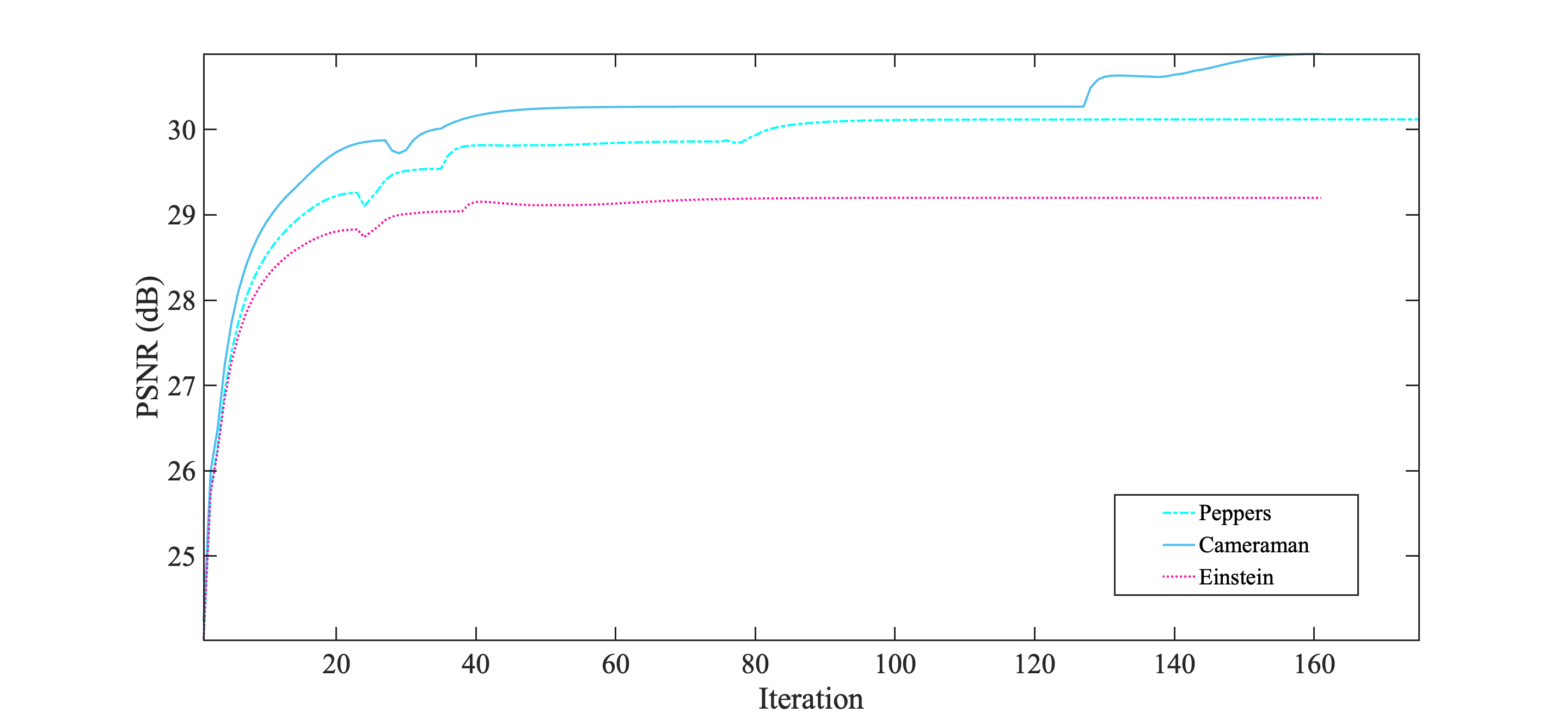}}\\
\subfloat[]{\includegraphics[width=0.8\textwidth]{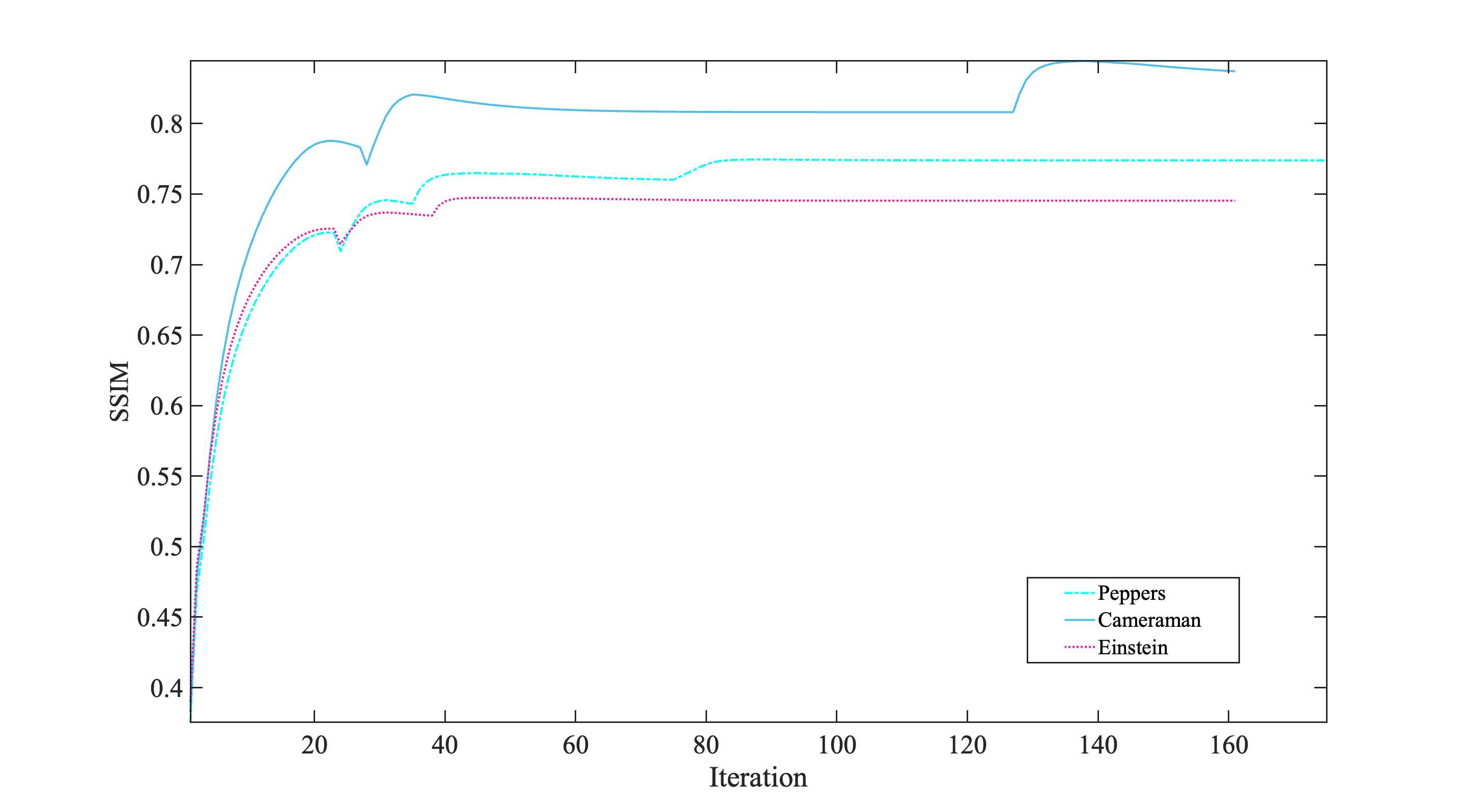}}
\caption{A visual representation of convergence analysis.(a) PSNR curves for the Peppers, Cameraman, and Einstein images; (b) SSIM curves for the same images.}
\label{PSNR and SSIM curve}
\end{figure}

In the proposed alternating optimization framework, the regularization terms $\lambda_1 \mathrm{TV}_1(u)$ and $\lambda_2 \mathrm{TV}_2(u)$ are not explicitly combined in a single objective function. Instead, each term is activated in a separate stage, and their relative influence is determined by an adaptive selection function $\sigma(i)$. Specifically, at each stage, both first- and higher- order TV sub-problems are solved independently. The restored result with the higher PSNR values is retained as the intermediate local optimal solution and used as an input for the next stage. This implicit selection mechanism allows the framework to dynamically balance edge preservation without requiring manual tuning of the relative weights  $\lambda_1$  and $ \lambda_2$. In the numerical implementation, each sub-problem has its own regularization parameter $\lambda$ used.  

An appropriate choice of $\lambda$ helps suppress artifacts while preserving sharp and natural textures. In our experiments, due to the alternating nature of our framework, a smaller $\lambda$ leads to better PSNR performance. therefore, we empirically set $\lambda = 2$. For the update of the parameter $\beta$, we adopt their original update strategy. Regarding the selection of the parameter $\alpha$, we empirically choose $\alpha = 0.618$.

\begin{figure}[htbp]
\centering
\includegraphics[width=0.8\textwidth]{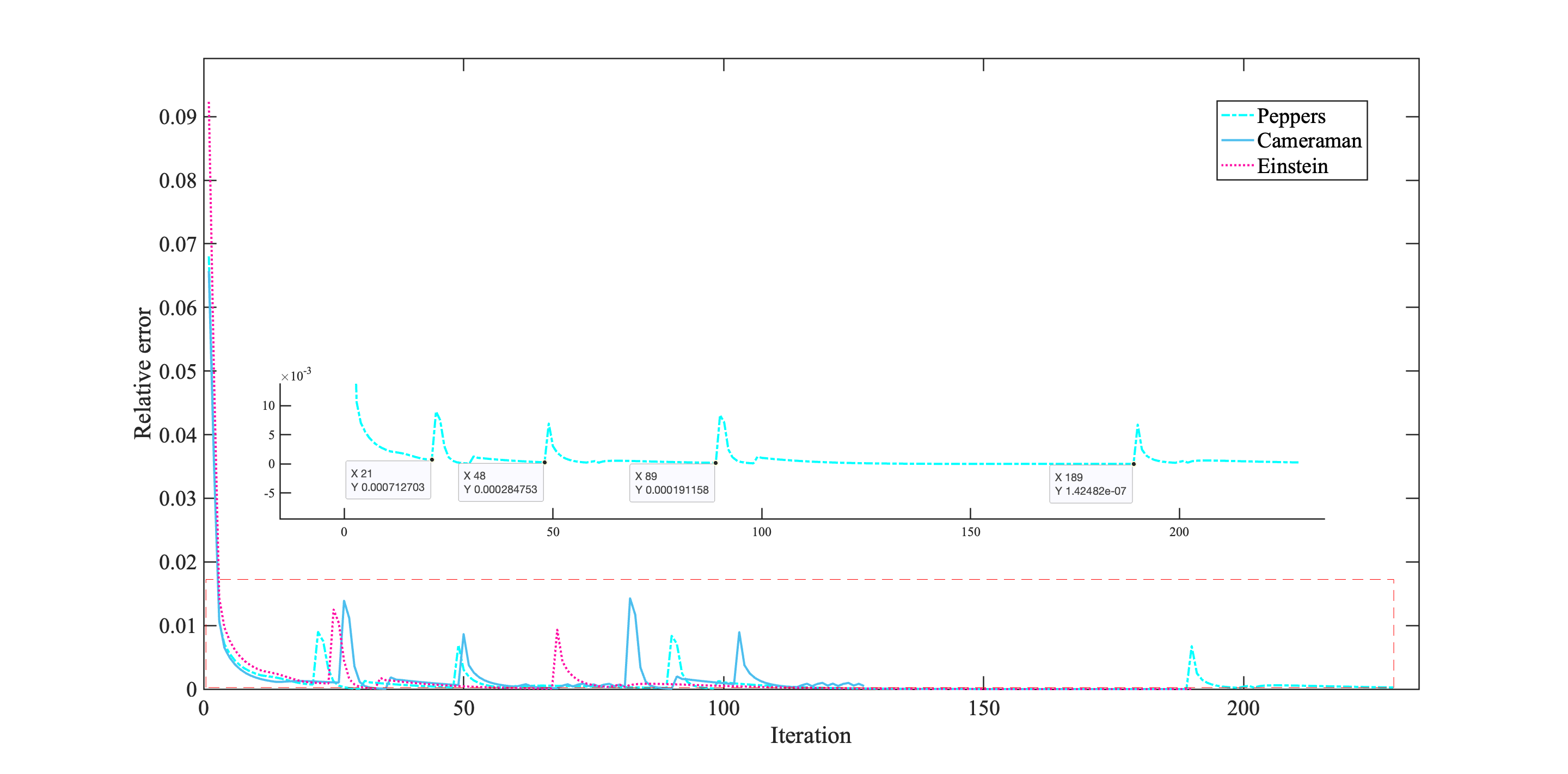}
\caption{Relative error curve for the Peppers, Cameraman, and Einstein images.} 
\label{Relative error curve}
\end{figure}   

To illustrate graphically the convergence behavior of the proposed stage-wise alternating method, we present the progression of PSNR, SSIM, and relative error metrics over the optimization iterations when $\sigma = 20$, as shown in Figures~\ref{PSNR and SSIM curve} and \ref{Relative error curve}. Although experiments were conducted on all images shown in Figure~\ref{tested image}, only the Cameraman, Peppers, and Einstein images were selected as representative test cases for their rich textural details.

\begin{figure}[htbp]
\centering
\setlength{\fboxsep}{0pt} 
\setlength{\fboxrule}{0.8pt} 
\subfloat[]{
    \begin{overpic}[width=0.22\textwidth]{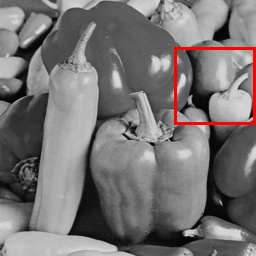}
        \put(0,0){%
            \fcolorbox{red}{white}{%
                \includegraphics[width=0.1\textwidth]{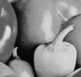}%
            }%
        }
    \end{overpic}
} 
\subfloat[]{%
    \begin{overpic}[width=0.22\textwidth]{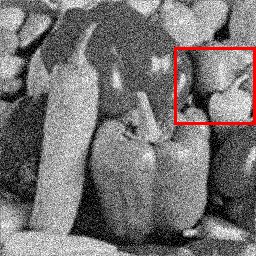}
        \put(0,0){%
            \fcolorbox{red}{white}{%
                \includegraphics[width=0.1\textwidth]{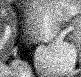}%
            }%
        }
    \end{overpic}
} 
\subfloat[]{%
    \begin{overpic}[width=0.22\textwidth]{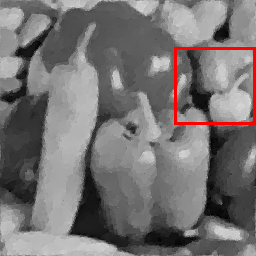}
        \put(0,0){%
            \fcolorbox{red}{white}{%
                \includegraphics[width=0.1\textwidth]{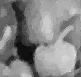}%
            }%
        }
    \end{overpic}
} 
\subfloat[]{%
    \begin{overpic}[width=0.22\textwidth]{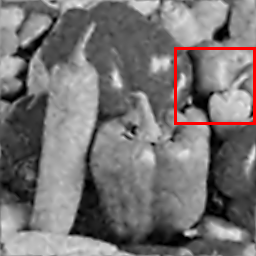}
        \put(0,0){%
            \fcolorbox{red}{white}{%
                \includegraphics[width=0.1\textwidth]{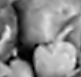}%
            }%
        }
    \end{overpic}
}
\\
\subfloat[]{%
    \begin{overpic}[width=0.22\textwidth]{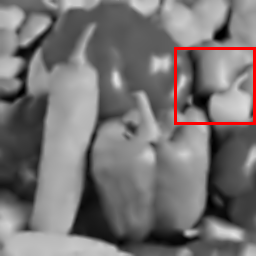}
        \put(0,0){%
            \fcolorbox{red}{white}{%
                \includegraphics[width=0.1\textwidth]{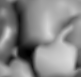}%
            }%
        }
    \end{overpic}
} 
\subfloat[]{%
    \begin{overpic}[width=0.22\textwidth]{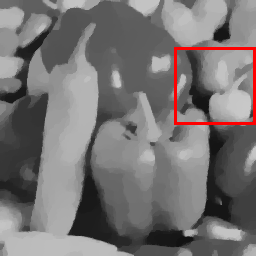}
        \put(0,0){%
            \fcolorbox{red}{white}{%
                \includegraphics[width=0.1\textwidth]{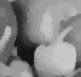}%
            }%
        }
    \end{overpic}
} 
\subfloat[]{%
    \begin{overpic}[width=0.22\textwidth]{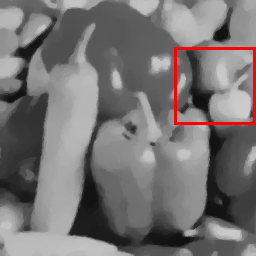}
        \put(0,0){%
            \fcolorbox{red}{white}{%
                \includegraphics[width=0.1\textwidth]{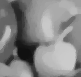}%
            }%
        }
    \end{overpic}
} 
\subfloat[]{%
    \begin{overpic}[width=0.22\textwidth]{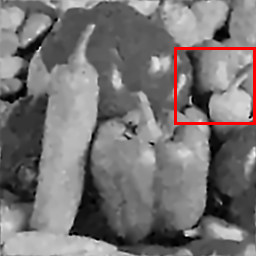}
        \put(0,0){%
            \fcolorbox{red}{white}{%
                \includegraphics[width=0.1\textwidth]{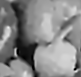}%
            }%
        }
    \end{overpic}
}
\\[-0.3cm]
\caption{Restored Peppers. (a)Original, (b)Degraded, (c)ROF model, (d)LLT model, (e)TGV model, (f)DBC-TV model, (g)Nesterov's method, and (h)Proposed method.}
\label{Peppers_marked}
\end{figure}

As observed, the PSNR and SSIM curves exhibit a slight temporary decline after a certain number of iterations, indicating the alternation of the regularizer. This is subsequently followed by a continued increase until a global optimum is reached. Empirical results indicate that the optimum is generally achieved after the $5-6$ stages. Moreover, the relative error gradually decreases as the number of stages increases, ultimately demonstrating stable convergence.

\begin{figure}[htbp]
\centering
\setlength{\fboxsep}{0pt} 
\setlength{\fboxrule}{0.8pt} 

\subfloat[]{%
    \begin{overpic}[width=0.22\textwidth]{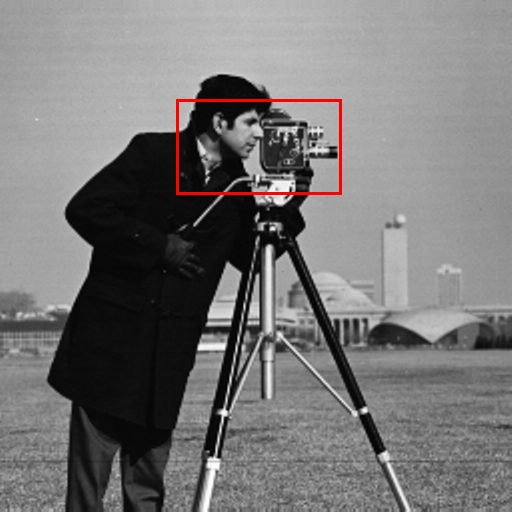}
        \put(0,0){%
            \fcolorbox{red}{white}{%
                \includegraphics[width=0.1\textwidth]{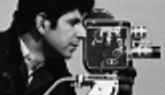}%
            }%
        }
    \end{overpic}
} \hspace{0.1cm}
\subfloat[]{%
    \begin{overpic}[width=0.22\textwidth]{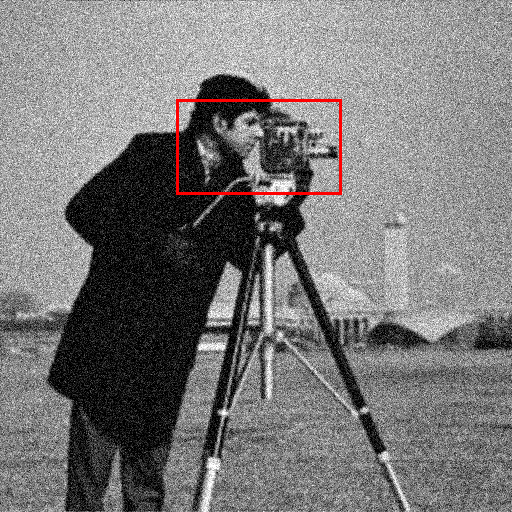}
        \put(0,0){%
            \fcolorbox{red}{white}{%
                \includegraphics[width=0.1\textwidth]{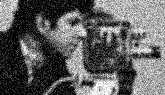}%
            }%
        }
    \end{overpic}
} \hspace{0.1cm}
\subfloat[]{%
    \begin{overpic}[width=0.22\textwidth]{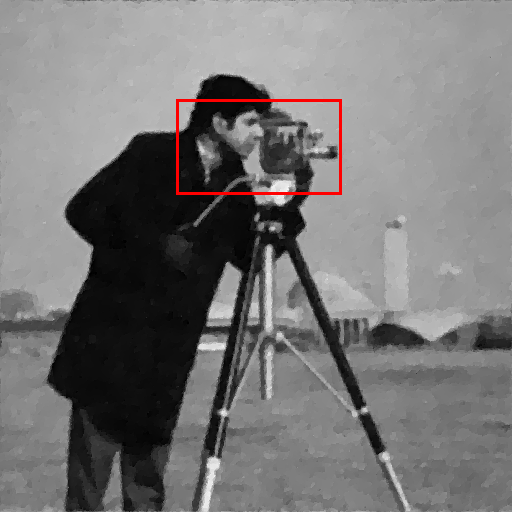}
        \put(0,0){%
            \fcolorbox{red}{white}{%
                \includegraphics[width=0.1\textwidth]{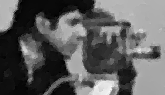}%
            }%
        }
    \end{overpic}
} \hspace{0.1cm}
\subfloat[]{%
    \begin{overpic}[width=0.22\textwidth]{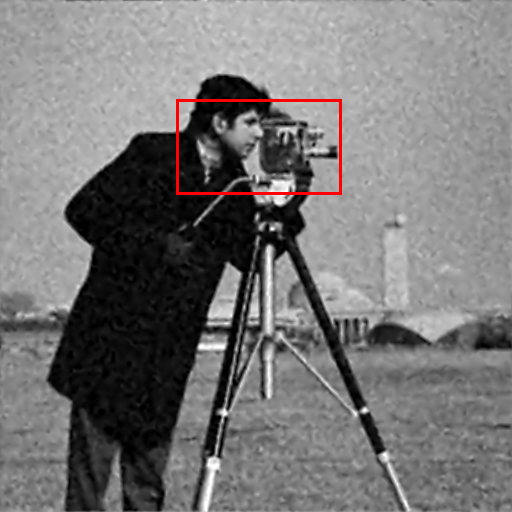}
        \put(0,0){%
            \fcolorbox{red}{white}{%
                \includegraphics[width=0.1\textwidth]{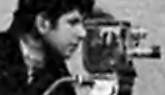}%
            }%
        }
    \end{overpic}
} \hspace{0.1cm}
\\
\subfloat[]{%
    \begin{overpic}[width=0.22\textwidth]{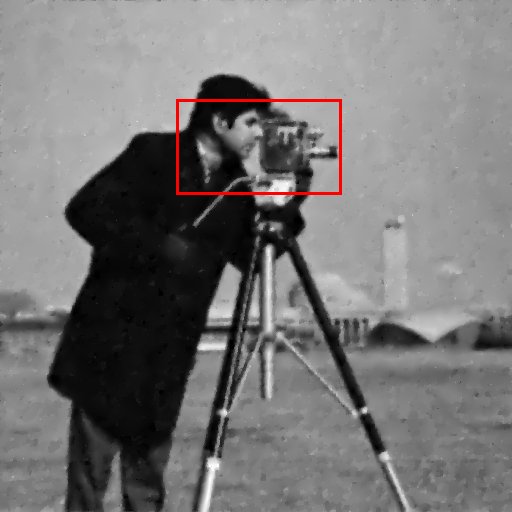}
        \put(0,0){%
            \fcolorbox{red}{white}{%
                \includegraphics[width=0.1\textwidth]{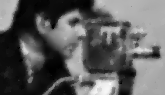}%
            }%
        }
    \end{overpic}
} \hspace{0.1cm}
\subfloat[]{%
    \begin{overpic}[width=0.22\textwidth]{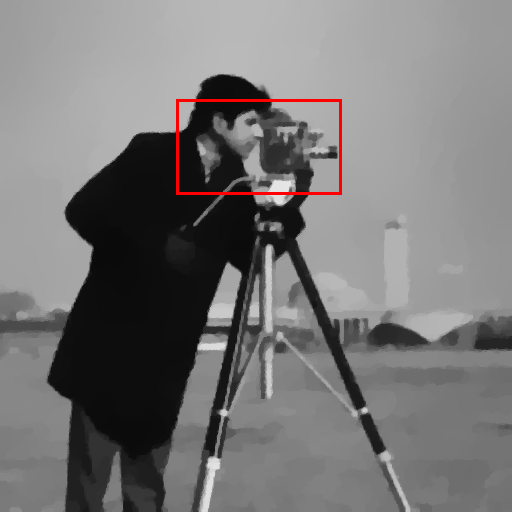}
        \put(0,0){%
            \fcolorbox{red}{white}{%
                \includegraphics[width=0.1\textwidth]{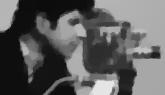}%
            }%
        }
    \end{overpic}
} \hspace{0.1cm}
\subfloat[]{%
    \begin{overpic}[width=0.22\textwidth]{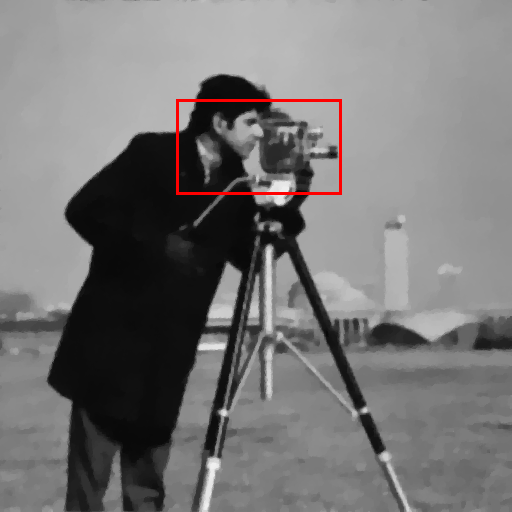}
        \put(0,0){%
            \fcolorbox{red}{white}{%
                \includegraphics[width=0.1\textwidth]{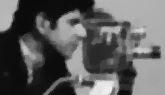}%
            }%
        }
    \end{overpic}
} \hspace{0.1cm}
\subfloat[]{%
    \begin{overpic}[width=0.22\textwidth]{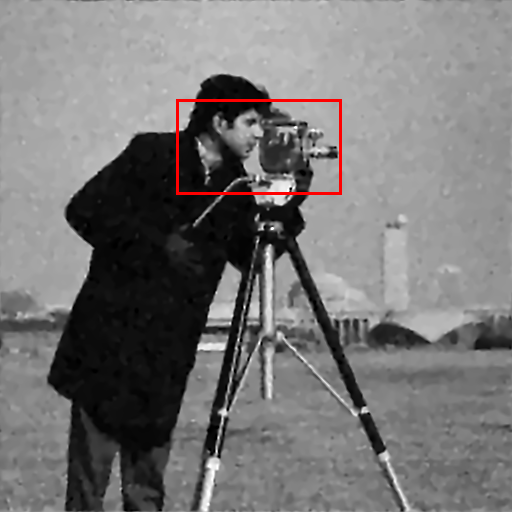}
        \put(0,0){%
            \fcolorbox{red}{white}{%
                \includegraphics[width=0.1\textwidth]{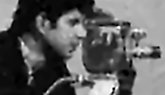}%
            }%
        }
    \end{overpic}
}
\\[-0.3cm]
\caption{Restored Cameraman. (a)Original, (b)Degraded, (c)ROF model, (d)LLT model, (e)TGV model, (f)DBC-TV model, (g)Nesterov's method, and (h)Proposed method.}
\label{Cameraman_marked}
\end{figure} 
Additionally, minor fluctuations in the relative error curve are observed throughout the iterative process, primarily resulting from differences in update strategies between the first-order TV regularization term and the higher-order TV regularization term. These fluctuations are particularly evident during the transition phase, where the TV model initially undergoes a brief adjustment period to accommodate newly imposed constraints. However, these transient fluctuations do not impede the overall convergence behavior. Rather, they reflect the algorithm’s adaptive capacity to adjust and alternate between the two TV models, ensuring effective optimization. 

\begin{figure}[htbp]
\centering
\setlength{\fboxsep}{0pt} 
\setlength{\fboxrule}{0.8pt} 

\subfloat[]{%
    \begin{overpic}[width=0.22\textwidth]{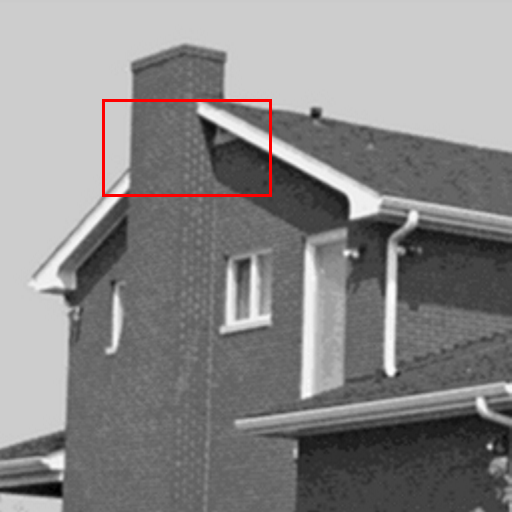}
        \put(0,0){%
            \fcolorbox{red}{white}{%
                \includegraphics[width=0.1\textwidth]{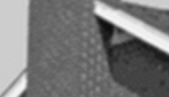}%
            }%
        }
    \end{overpic}
} \hspace{0.1cm}
\subfloat[]{%
    \begin{overpic}[width=0.22\textwidth]{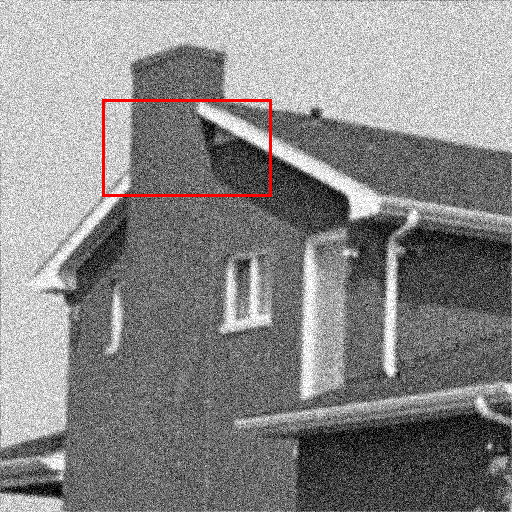}
        \put(0,0){%
            \fcolorbox{red}{white}{%
                \includegraphics[width=0.1\textwidth]{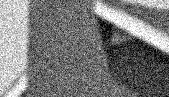}%
            }%
        }
    \end{overpic}
} \hspace{0.1cm}
\subfloat[]{%
    \begin{overpic}[width=0.22\textwidth]{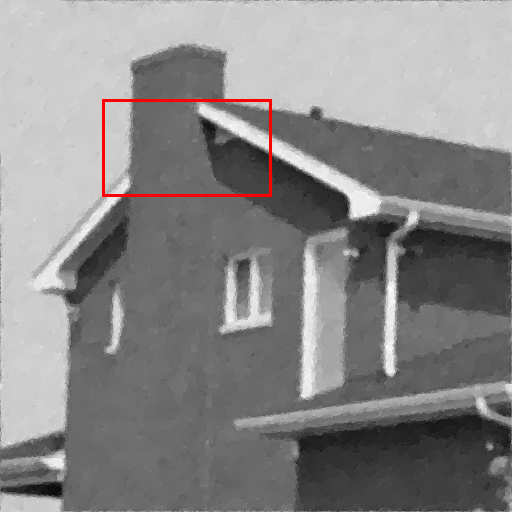}
        \put(0,0){%
            \fcolorbox{red}{white}{%
                \includegraphics[width=0.1\textwidth]{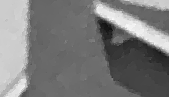}%
            }%
        }
    \end{overpic}
} \hspace{0.1cm}
\subfloat[]{%
    \begin{overpic}[width=0.22\textwidth]{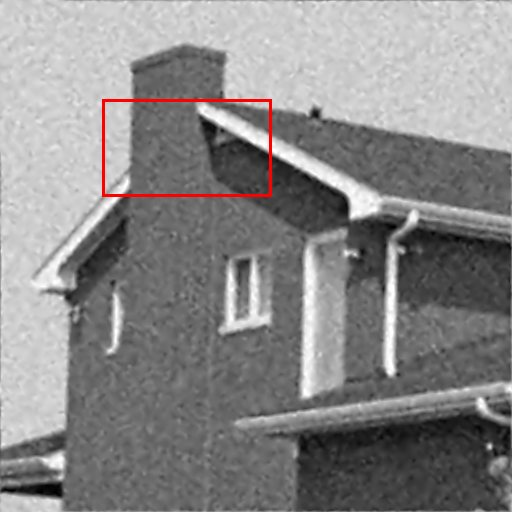}
        \put(0,0){%
            \fcolorbox{red}{white}{%
                \includegraphics[width=0.1\textwidth]{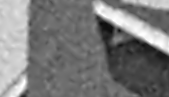}%
            }%
        }
    \end{overpic}
}
\\
\subfloat[]{%
    \begin{overpic}[width=0.22\textwidth]{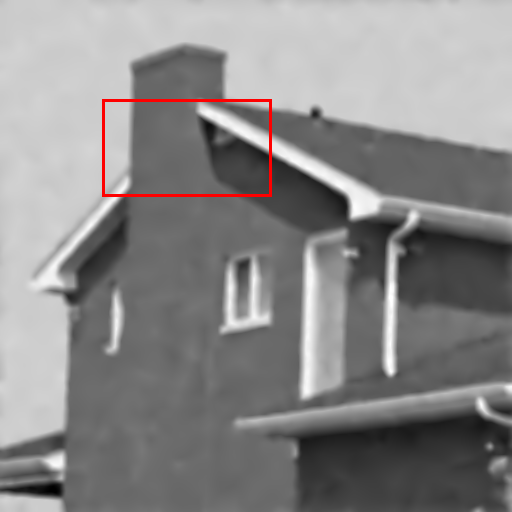}
        \put(0,0){%
            \fcolorbox{red}{white}{%
                \includegraphics[width=0.1\textwidth]{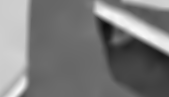}%
            }%
        }
    \end{overpic}
} \hspace{0.1cm}
\subfloat[]{%
    \begin{overpic}[width=0.22\textwidth]{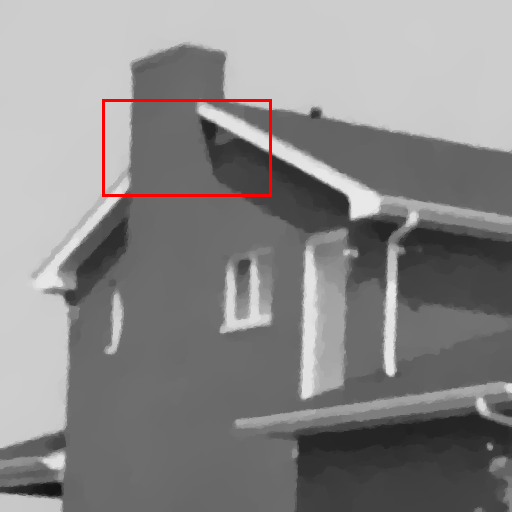}
        \put(0,0){%
            \fcolorbox{red}{white}{%
                \includegraphics[width=0.1\textwidth]{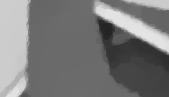}%
            }%
        }
    \end{overpic}
} \hspace{0.1cm}
\subfloat[]{%
    \begin{overpic}[width=0.22\textwidth]{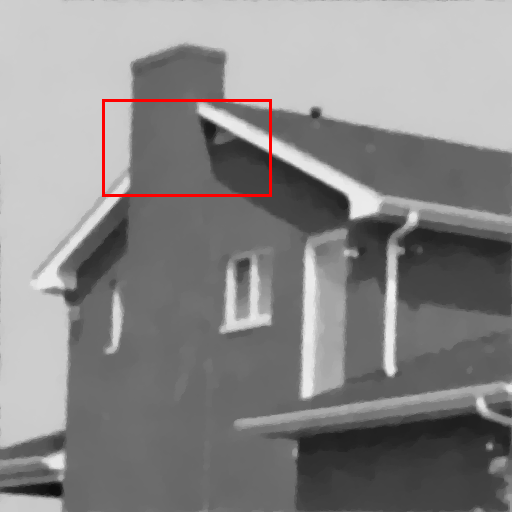}
        \put(0,0){%
            \fcolorbox{red}{white}{%
                \includegraphics[width=0.1\textwidth]{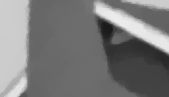}%
            }%
        }
    \end{overpic}
} \hspace{0.1cm}
\subfloat[]{%
    \begin{overpic}[width=0.22\textwidth]{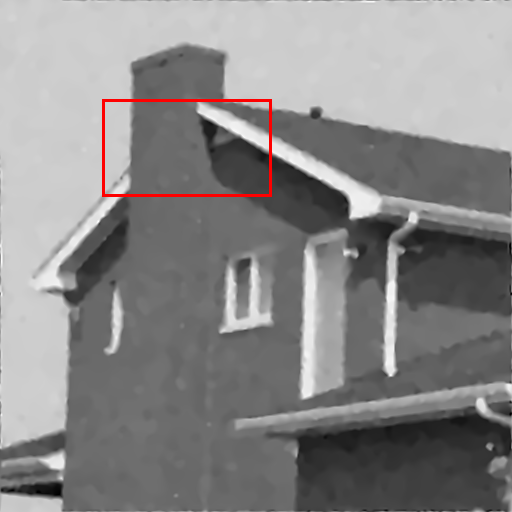}
        \put(0,0){%
            \fcolorbox{red}{white}{%
                \includegraphics[width=0.1\textwidth]{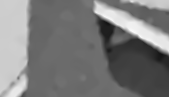}%
            }%
        }
    \end{overpic}
}
\\[-0.3cm]
\caption{Restored House. (a)Original, (b)Degraded, (c)ROF model, (d)LLT model, (e) TGV model, (f)DBC-TV model, (g)Nesterov's method, and (h)Proposed method.}
\label{House_marked}
\end{figure}

Interestingly, in the case of the Peppers image, the first stage of the restoration procedure converges at the $21$st iteration. During the subsequent iterations, the relative error curve exhibits temporary increases. Then, the second and third stages are completed at the $48$th and $89$th iterations, respectively, and the process ultimately converges at the $189$th iteration. It is worth noting that the relative error at each stage becomes progressively lower than in the previous one, providing evidence for the convergence of the proposed method
, which follows the theoretical convergence analysis shown in \cref{fig:tikz}.

Regarding the Cameraman image in Figure~\ref{Cameraman_marked}, which contains less edge information, all methods demonstrated effective denoising and satisfactory edge preservation, except for the ROF model~\cite{rudin1992nonlinear}, which exhibited residual artifacts, particularly on the cameraman’s face. Comparative analysis indicates that the LLT model~\cite{lysaker2003noise} achieves superior edge preservation. However, our proposed method yields comparable performance in maintaining the important edge.

Meanwhile, the House image in Figure~\ref{House_marked} features complex and rich edge structures and textures. Inspection of the zoomed-in region shows that all methods achieved effective denoising performance, although the ROF model~\cite{rudin1992nonlinear} displayed residual artifacts. Comparative analysis of restored enlarged images reveals that, while the TGV model~\cite{bredies2010total} introduces slight blurring, it achieves the best overall restoration by effectively balancing denoising and edge preservation. Visually, our proposed method also produces satisfactory results, underscoring its competitiveness in maintaining a favorable trade-off between noise suppression and detail retention.

\begin{figure}[htbp]
\centering
\setlength{\fboxsep}{0pt}
\setlength{\fboxrule}{0.8pt}

\subfloat[]{%
    \begin{overpic}[width=0.22\textwidth]{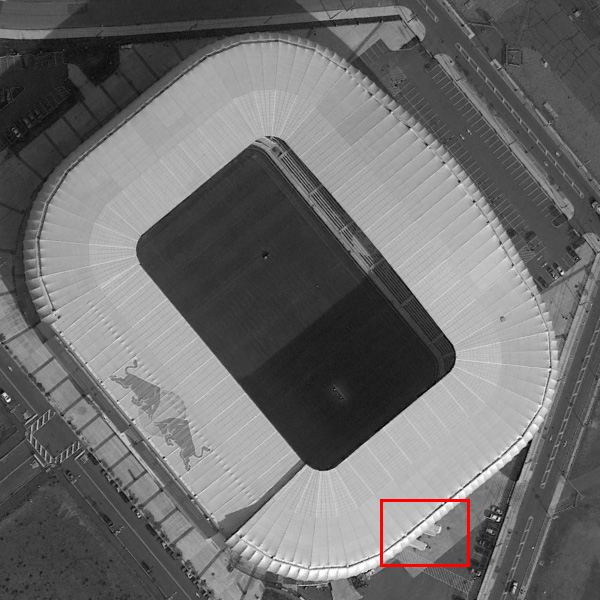}
        \put(0,0){%
            \fcolorbox{red}{white}{\includegraphics[width=0.1\textwidth]{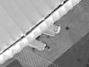}}%
        }
    \end{overpic}
} \hspace{0.1cm}
\subfloat[]{%
    \begin{overpic}[width=0.22\textwidth]{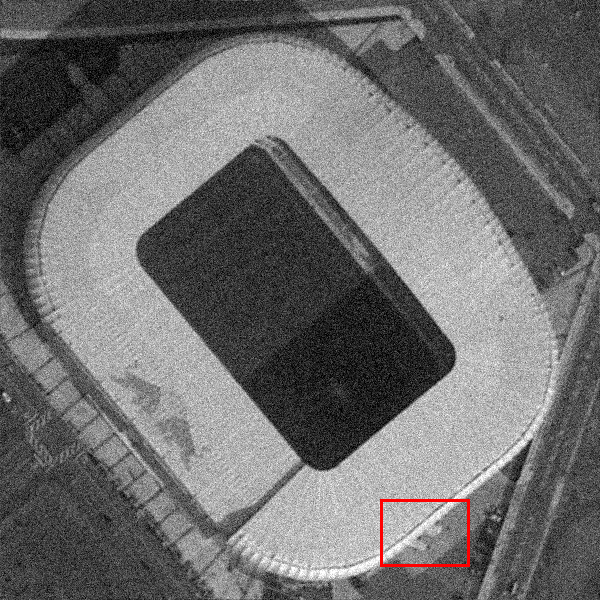}
        \put(0,0){%
            \fcolorbox{red}{white}{\includegraphics[width=0.1\textwidth]{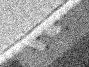}}%
        }
    \end{overpic}
} \hspace{0.1cm}
\subfloat[]{%
    \begin{overpic}[width=0.22\textwidth]{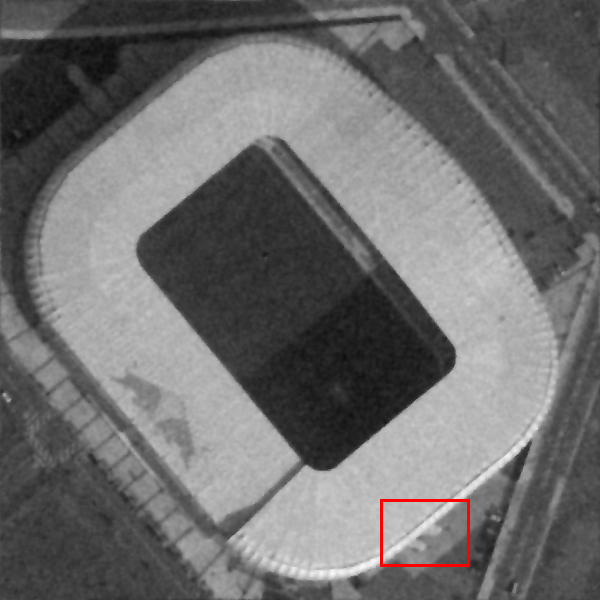}
        \put(0,0){%
            \fcolorbox{red}{white}{\includegraphics[width=0.1\textwidth]{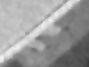}}%
        }
    \end{overpic}
} \hspace{0.1cm}
\subfloat[]{%
    \begin{overpic}[width=0.22\textwidth]{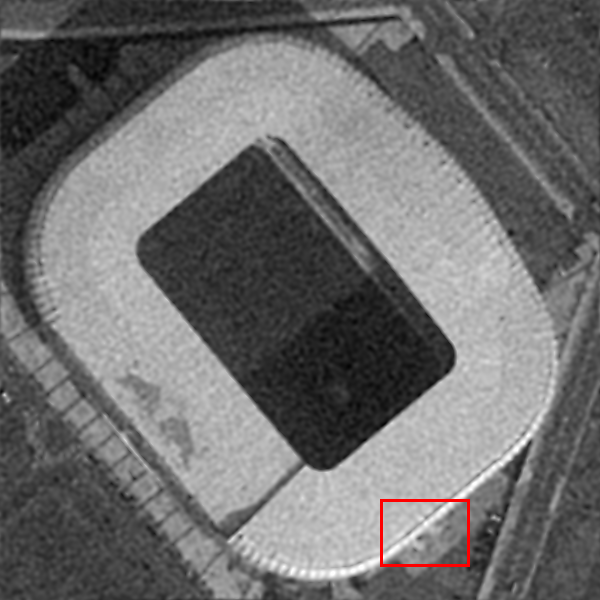}
        \put(0,0){%
            \fcolorbox{red}{white}{\includegraphics[width=0.1\textwidth]{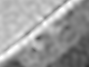}}%
        }
    \end{overpic}
}
\\
\subfloat[]{%
    \begin{overpic}[width=0.22\textwidth]{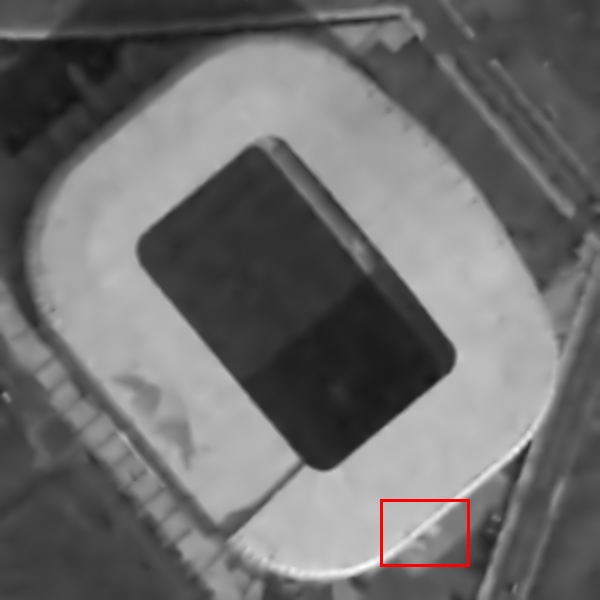}
        \put(0,0){%
            \fcolorbox{red}{white}{\includegraphics[width=0.1\textwidth]{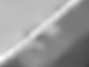}}%
        }
    \end{overpic}
} \hspace{0.1cm}
\subfloat[]{%
    \begin{overpic}[width=0.22\textwidth]{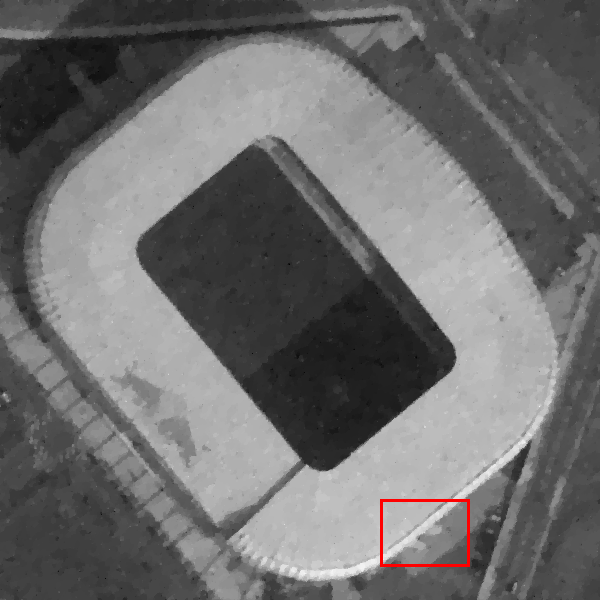}
        \put(0,0){%
            \fcolorbox{red}{white}{\includegraphics[width=0.1\textwidth]{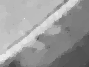}}%
        }
    \end{overpic}
} \hspace{0.1cm}
\subfloat[]{%
    \begin{overpic}[width=0.22\textwidth]{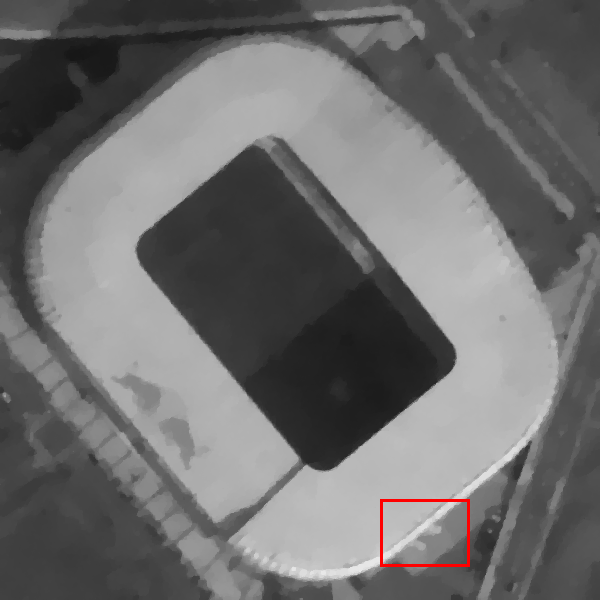}
        \put(0,0){%
            \fcolorbox{red}{white}{\includegraphics[width=0.1\textwidth]{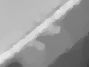}}%
        }
    \end{overpic}
} \hspace{0.1cm}
\subfloat[]{%
    \begin{overpic}[width=0.22\textwidth]{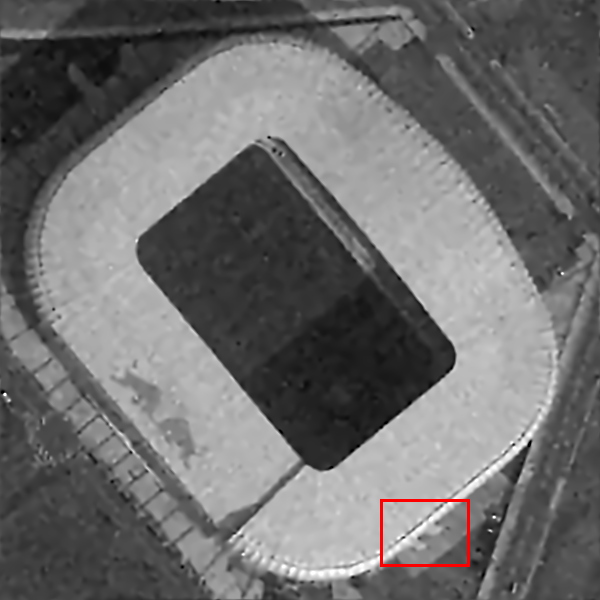}
        \put(0,0){%
            \fcolorbox{red}{white}{\includegraphics[width=0.1\textwidth]{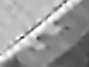}}%
        }
    \end{overpic}
}
\\[-0.3cm]
\caption{Restored Stadium. (a)Original, (b)Degraded, (c)ROF model, (d)LLT model, (e)TGV model, (f)DBC-TV model, (g)Nesterov's method, and (h)Proposed method.}
\label{Stadium_marked}
\end{figure}

Finally, examination of the more complex stadium image in Figure~\ref{Stadium_marked} reveals that all methods deliver suboptimal denoising performance, resulting in image blurring and loss of details. Nonetheless, our proposed method demonstrates superior capability in restoring image details, preserving object contours, and notably enhancing the visibility of features such as the drain pipe and the chair.

\begin{table}[htbp]
\caption{The restored results(PSNRs,SSIMs and FOMs) for $3\times3$ Gaussian blur kernel at noise levels $\sigma=10$.}
\label{table_10}
\centering
\renewcommand{\arraystretch}{1.5} 
\resizebox{1\textwidth}{!}{  
\begin{tabular}{l c c c c c c c}
\cline{1-8}
&\textbf{Image} & \textbf{ROF} & \textbf{LLT} & \textbf{TGV} & \textbf{DBC-TV} & \textbf{Nesterov} &\textbf{Proposed} \\
\cline{2-8}
\multirow{12}{*}{$\sigma = 10$} 
& Cameraman & 29.506/0.861/0.548 & 31.322/0.891/0.673 & 28.721/0.845/0.463 & 27.871/0.820/0.433 &  30.627/0.883/0.668 &
\textbf{34.161/0.926/0.778} \\ 
& Einstein & 28.679/0.715/0.558 & 29.527/0.759/0.635 & 28.449/0.701/0.444 & 27.459/0.670/0.426 &  29.679/0.747/0.631 &\textbf{31.069/0.805/0.757} \\ 
& Woman & 33.450/0.893/\textbf{0.618} & 34.731/0.917/0.592 & 34.160/0.896/0.432 & 32.150/0.882/0.524 & 34.375/0.900/0.590 & \textbf{35.589/0.919}/0.557 \\ 
& Peppers & 27.795/0.851/0.762 & 28.592/0.877/0.806 & 26.909/0.827/0.668 & 26.607/0.818/0.700 & 30.234/0.782/0.694 &\textbf{31.363/0.913/0.877} \\ 
&Butterfly & 23.586/0.769/0.751 & 24.364/0.815/0.804 & 22.804/0.742/0.672 & 22.835/0.703/0.665 & 25.751/0.848/0.829 &\textbf{27.033/0.881/0.931} \\ 
& Starfish & 24.757/0.745/0.725 & 25.959/0.796/0.754 & 24.238/0.709/0.622 & 23.337/0.671/0.637 & 25.644/0.779/0.772 &\textbf{28.491/0.870/0.912} \\ 
& Boat & 26.510/0.696/0.573 & 27.389/0.734/0.645 & 25.838/0.664/0.472 & 25.313/0.644/0.455 &  27.261/0.728/0.643 &\textbf{29.359/0.799/0.822} \\ 
& House & 32.602/0.901/0.404 & 35.268/0.941/\textbf{0.646} & 33.394/0.897/0.332 & 31.357/0.878/0.327 & 34.083/0.912/0.421 & \textbf{35.817/0.943}/0.523 \\ 
&Text & 18.560/0.728/0.744 & 18.977/0.711/0.756 & 18.286/0.771/0.746 & 18.846/0.826/0.754 &  19.609/0.872/0.777 &\textbf{22.681/0.873/0.828} \\ 
&Logan & 23.644/0.818/0.734 & 23.533/0.785/0.451 & 23.197/0.879/0.937& 24.021/0.903/\textbf{0.978} & 24.241/0.905/0.964 & \textbf{25.670/0.917}/0.762 \\ 
&Stadium & 27.900/0.756/0.557 & 27.890/0.747/\textbf{0.636} & 27.137/0.709/0.338 & 27.616/0.729/0.444 & 28.743/0.768/0.504 & \textbf{29.354/0.780}/0.548 \\ 
&Brain & 28.435/0.726/0.549 & 28.350/0.698/\textbf{0.625} & 27.470/0.745/0.314 &28.425/0.764/0.366 & 29.281/0.800/0.470 & \textbf{30.642/0.802}/0.500 \\ 
\cline{1-8}
\end{tabular}
}
\end{table}

The foregoing analysis demonstrates the effectiveness and convergence of our proposed method. We further validate it through objective quantitative evaluations. These observations are supported by the results summarized in Tables~\ref{table_10}-\ref{table_30}. In these tables, the best values in each row are highlighted in boldface. Our method consistently achieves superior PSNR and competitive SSIM and FOM values across various noise levels ($\sigma = 10, 20, 30$). For instance, at $\sigma = 10$, our method demonstrates superior performance over all competing methods, resulting in an average PSNR improvement of approximately $1.8$ dB and an SSIM increase of $0.04$. However, it exhibits slightly lower FOM values on images with fewer structural details and edge information, such as House, Logan, Stadium, and Brain, likely due to over-smoothing of subtle features. This leads to a minor reduction in FOM values.

\begin{table}[H]
\caption{The denoising results(PSNRs,SSIMs and FOMs) for $3\times3$ Gaussian blur kernel at noise levels $\sigma=20$.}
\label{table_20}

\centering
\renewcommand{\arraystretch}{1.5} 
\resizebox{1\textwidth}{!}{  
\begin{tabular}{l c c c c c c c}

\cline{1-8}

 & \textbf{Image} & \textbf{ROF} & \textbf{LLT} & \textbf{TGV} & \textbf{DBC-TV} & \textbf{Nesterov's method} &\textbf{Proposed} \\

\cline{2-8}

\multirow{12}{*}{$\sigma = 20$} 
& Cameraman & 28.983/0.848/0.669 & 30.403/0.857/0.701 & 28.348/0.835/0.480 & 27.855/0.820/0.453 & 28.679/0.844/0.533 &\textbf{31.413/0.881/0.760} \\ 
& Einstein & 28.276/0.699/0.651 & 28.966/0.747/0.677 & 28.322/0.701/0.460 & 27.471/0.667/0.438 &28.384/0.694/0.514 & \textbf{29.421/0.749/0.695} \\ 
& Woman & 32.041/0.850/0.676 & 33.032/0.882/0.664 & 33.808/0.896/0.466 & 31.947/0.874/0.556 & 32.514/0.877/\textbf{0.577} &\textbf{33.946/0.897}/0.493 \\ 
& Peppers & 27.019/0.811/0.801 & 28.002/0.851/0.811 & 26.678/0.822/0.668 & 26.518/0.810/0.709 &  28.741/0.751/0.662 &\textbf{29.221/0.869/0.873} \\ 
& Butterfly & 23.281/0.734/0.784 & 24.111/0.794/0.802 & 22.727/0.732/0.647 & 22.845/0.700/0.671 &  24.275/0.794/0.785 &\textbf{25.211/0.806/0.892} \\ 
& Starfish & 24.446/0.720/0.738 & 25.600/0.783/0.801 & 24.046/0.700/0.648 & 23.342/0.669/0.623 &  24.124/0.708/0.699 &\textbf{26.346/0.799/0.873} \\ 
& Boat & 26.196/0.681/0.638 & 26.994/0.719/0.689 & 25.715/0.657/0.461 & 25.346/0.643/0.484 & 25.970/0.671/0.565 &\textbf{27.717/0.741/0.772} \\ 
& House & 31.508/0.882/0.511 & 32.713/0.894/\textbf{0.586} & 32.032/0.883/0.308 & 31.201/0.875/0.314 &32.172/0.885/0.355 & \textbf{33.567/0.912}/0.476 \\ 
&Text & 18.433/0.561/0.747 &18.886/0.537/0.757 & 18.308/0.783/0.744 & 18.770/\textbf{0.808}/0.771 &18.925/0.796/0.763 & \textbf{21.969}/0.638/\textbf{0.828} \\ 
&Logan & 23.358/0.695/0.518 &23.308/0.652/0.320 & 23.140/0.849/\textbf{0.950} & 23.852/0.870/0.933 &23.734/\textbf{0.874}/0.949 & \textbf{25.092}/0.764/0.452 \\ 
&Stadium & 27.487/0.717/0.674 &27.297/0.697/0.672 & 27.296/0.718/0.372 & 27.589/0.719/0.367 &27.566/0.729/0.435 & \textbf{28.607/0.760/0.694} \\ 
&Brain & 27.833/0.613/0.481 &27.696/0.566/\textbf{0.614} & 27.540/0.713/0.356 & 28.188/0.727/0.399 &28.149/\textbf{0.743}/0.381 & \textbf{29.653}/0.682/0.504 \\ 
\cline{1-8}

\end{tabular}
}
\end{table}

At higher noise levels ($\sigma = 20$ and $\sigma = 30$), although the PSNR advantage of our method persists, the margins narrow and the SSIM and FOM metrics reveal mixed performance compared to other approaches such as LLT, TGV, and DBC-TV. This trend suggests that while our method effectively balances noise reduction and edge preservation in moderately noisy conditions, its relative advantage diminishes under severe noise contamination, possibly reflecting the intrinsic limitations of the hybrid modeling approach.

\begin{table}[H]
\caption{The denoising results(PSNRs,SSIMs and FOMs) for $3\times3$ Gaussian blur kernel at noise levels $\sigma=30$.}
\label{table_30}
\centering
\renewcommand{\arraystretch}{1.5} 
\resizebox{1\textwidth}{!}{  
\begin{tabular}{l c c c c c c c}

\cline{1-8}
 & \textbf{Image} & \textbf{ROF} & \textbf{LLT} & \textbf{TGV} & \textbf{DBC-TV} & \textbf{Nesterov's method} &\textbf{Proposed} \\
\cline{2-8}

\multirow{12}{*}{$\sigma = 30$} 
& Cameraman & 27.721/0.813/0.617 & 28.628/0.758/0.603 & 27.282/0.810/0.405 & 27.826/0.820/0.437 &  27.536/0.820/0.479 &\textbf{29.892/0.842/0.713} \\ 
& Einstein & 27.268/0.659/0.609 & 27.729/0.687/\textbf{0.706} & 28.061/0.699/0.527 & 27.494/0.664/0.488 &27.666/0.671/0.520 & \textbf{28.278/0.711}/0.637 \\ 
& Woman & 30.661/0.817/\textbf{0.634} & 30.275/0.780/0.492 & 32.133/0.883/0.503 & 31.672/0.864/0.590 &31.750/0.867/0.558 & \textbf{32.367/0.886}/0.559 \\ 
& Peppers & 25.677/0.762/0.792 & 26.778/0.783/0.828 & 25.795/0.790/0.653 & 26.391/0.802/0.718 &  27.014/0.737/0.665 &\textbf{27.504/0.828/0.841} \\ 
& Butterfly & 22.343/0.676/0.743 & 23.566/0.746/0.829 & 22.698/0.727/0.715 & 22.853/0.697/0.679 & 23.246/0.760/0.774 &\textbf{23.914/0.766/0.835} \\ 
& Starfish & 23.428/0.666/0.753 & 24.858/0.741/0.806 & 23.051/0.651/0.574 & 23.359/0.667/0.651 & 23.180/0.657/0.665 & \textbf{24.911/0.746/0.866} \\ 
& Boat & 25.316/0.640/0.631 & 26.144/0.670/\textbf{0.762} & 25.082/0.635/0.443 & 25.389/0.643/0.484 &25.091/0.634/0.502 & \textbf{26.475/0.693}/0.700 \\ 
& House & 30.134/0.852/0.477 & 30.940/0.865/\textbf{0.512} & 31.595/0.871/0.473 & 30.994/0.871/0.325 &31.101/0.874/0.350 & \textbf{31.884/0.889}/0.375 \\
&Text & 18.296/0.445/0.751 & 18.718/0.422/0.765 & 18.162/0.598/0.749 & 18.600/\textbf{0.997}/0.756 &18.291/0.723/0.743 & \textbf{21.033}/0.484/\textbf{0.824} \\
&Logan & 23.143/0.612/0.450 &22.986/0.543/0.276 & 22.891/0.692/0.591 & 23.504/\textbf{0.861}/0.930 &23.226/0.843/\textbf{0.939}& \textbf{24.414}/0.679/0.455 \\
&Stadium & 26.892/0.666/0.640 &26.892/0.666/\textbf{0.640} & 27.089/0.697/0.564 & 27.126/0.717/0.386 &26.889/0.709/0.398 & \textbf{27.518/0.722}/0.529 \\
&Brain & 27.112/0.519/0.597 &26.821/0.474/\textbf{0.567} & 27.274/0.599/0.521 & 27.562/\textbf{0.725}/0.388 &27.190/0.696/0.351& \textbf{28.475}/0.597/0.482 \\

\cline{1-8}

\end{tabular}
}
\end{table} 

Overall, across all three noise levels, our proposed method consistently achieves the highest PSNR values, with average gains of approximately 1.8 dB, 1.26 dB, and 1.03 dB over competing TV-based approaches for all tested images. It also attains superior SSIM performance in most cases, particularly for images characterized by sharp edges and rich textures, such as Peppers, Butterfly, and Starfish. Conversely, it exhibits marginally lower FOM values on smoother images (e.g., House, Logan, Stadium, Brain), especially at elevated noise levels. Under the most severe noise condition ($\sigma = 30$), both SSIM and FOM metrics indicate inferior performance relative to LLT, TGV, and DBC-TV models, suggesting that the method’s effectiveness in preserving precise edge structures diminishes as noise level increases.

These findings collectively validate the effectiveness and denoising capability of our approach across a diverse range of image types, demonstrating a consistent ability to suppress noise while preserving crucial edge and texture details. At the same time, the results delineate the constraints of the proposed method in maintaining fine edge preservation when confronted with extreme noise degradation, highlighting areas for potential future improvement.

\begin{table}[htbp]
\caption{The denoising results (PSNRs, SSIMs and FOMs) for $3\times3$ Gaussian blur kernel at noise levels $\sigma=20$.}
\label{table_net_20_combined}
\centering
\renewcommand{\arraystretch}{1.5}
\resizebox{1\textwidth}{!}{
\begin{tabular}{l c c c c c c}
\cline{1-7}
 & \textbf{Image} & \textbf{DRUNet} & \textbf{FFDNet} & \textbf{DnCNN} & \textbf{DPIR} & \textbf{Proposed} \\
\cline{2-7}

\multirow{12}{*}{$\sigma = 20$}

& Cameraman & 28.922 / 0.862 / \textbf{0.842} & 28.753 / 0.857 / 0.832 & 27.904 / 0.821 / 0.793 & 31.356 / \textbf{0.884} / 0.684 & \textbf{31.413} / 0.881 / 0.760 \\ 

& Einstein & 29.115 / 0.737 / \textbf{0.708} & 28.956 / 0.731 / 0.674 & 28.001 / 0.675 / 0.584 & \textbf{30.049} / \textbf{0.756} / 0.649 & 29.421 / 0.749 / 0.695 \\ 

& Woman & 34.574 / 0.903 / \textbf{0.752} & 34.304 / 0.899 / 0.711 & 33.366 / 0.884 / 0.681 & \textbf{34.987} / \textbf{0.905} / 0.533 & 33.946 / 0.897 / 0.493 \\ 

& Peppers & 29.592 / 0.775 / 0.841 & 29.485 / 0.772 / 0.824 & 28.897 / 0.754 / 0.787 & \textbf{30.817} / 0.790 / 0.801 & 29.221 / \textbf{0.869} / \textbf{0.873} \\ 

& Butterfly & 24.016 / 0.803 / 0.876 & 23.894 / 0.798 / 0.872 & 23.477 / 0.771 / 0.801 & \textbf{27.026} / \textbf{0.871} / 0.866 & 25.211 / 0.806 / \textbf{0.892} \\ 

& Starfish & 24.667 / 0.735 / 0.766 & 24.415 / 0.724 / 0.747 & 23.738 / 0.676 / 0.646 & 26.017 / 0.785 / 0.789 & \textbf{26.346} / \textbf{0.799} / \textbf{0.873} \\ 

& Boat & 26.224 / 0.695 / 0.710 & 26.094 / 0.688 / 0.700 & 25.349 / 0.637 / 0.624 & 27.407 / 0.730 / 0.669 & \textbf{27.717} / \textbf{0.741} / \textbf{0.772} \\ 

& House & 34.561 / 0.911 / 0.873 & 34.203 / 0.905 / \textbf{0.889} & 33.056 / 0.882 / 0.818 & \textbf{35.651} / \textbf{0.913} / 0.402 & 33.567 / 0.912 / 0.476 \\ 

& Text & 17.271 / 0.846 / 0.746 & 16.972 / 0.822 / 0.742 & 16.715 / 0.357 / 0.734 & \textbf{22.976} / \textbf{0.968} / \textbf{0.864} & 21.969 / 0.638 / 0.828 \\ 

& Logan & 22.516 / 0.915 / 0.824 & 22.536 / 0.906 / \textbf{0.974} & 22.358 / 0.589 / 0.973 & \textbf{27.158} / \textbf{0.970} / 0.966 & 25.092 / 0.764 / 0.452 \\ 

& Stadium & 27.711 / 0.740 / \textbf{0.713} & 27.577 / 0.730 / 0.669 & 26.868 / 0.697 / 0.518 & 28.601 / 0.758 / 0.533 & \textbf{28.607} / \textbf{0.760} / 0.694 \\ 

& Brain & 27.769 / 0.796 / \textbf{0.843} & 27.693 / 0.788 / 0.826 & 26.996 / 0.439 / 0.821 & \textbf{30.232} / \textbf{0.822} / 0.423 & 29.653 / 0.682 / 0.504 \\ 

\cline{2-7}
& \textbf{Average} & 27.245 / 0.810 / \textbf{0.791} & 27.074 / 0.802 / 0.788 & 26.567 / 0.706 / 0.731 & 28.693 / \textbf{0.845} / 0.681 & \textbf{28.935} / 0.791 / 0.682 \\ 
\cline{1-7}
\end{tabular}
}
\end{table}

Furthermore, we compare our method with several representative learning-based models, including DRUNet \cite{zhang2017learning}, FFDNet \cite{zhang2018ffdnet}, DnCNN \cite{zhang2017beyond}, and the Plug-and-Play method DPIR \cite{zhang2021plug}.

As shown in Table~\ref{table_net_20_combined}, the proposed method achieves the highest average PSNR value. Compared with learning-based models, our approach improves the average PSNR by approximately 1.86~dB, 1.69~dB, and 2.23~dB, respectively, while also achieving a 0.24~dB improvement over the DPIR method. The corresponding visual comparisons are presented in Figure~\ref{Text_net}.

\begin{figure}[htbp]
\centering
\setlength{\fboxsep}{0pt}
\setlength{\fboxrule}{0.8pt}

\subfloat[]{%
    \begin{overpic}[width=0.16\textwidth]{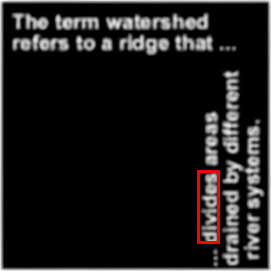}
        \put(0,0){%
            \fcolorbox{red}{white}{\includegraphics[width=0.03\textwidth]{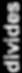}}%
        }
    \end{overpic}
} \hspace{0.1cm}
\subfloat[]{%
    \begin{overpic}[width=0.16\textwidth]{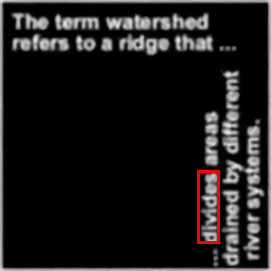}
        \put(0,0){%
            \fcolorbox{red}{white}{\includegraphics[width=0.03\textwidth]{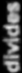}}%
        }
    \end{overpic}
} \hspace{0.1cm}
\subfloat[]{%
    \begin{overpic}[width=0.16\textwidth]{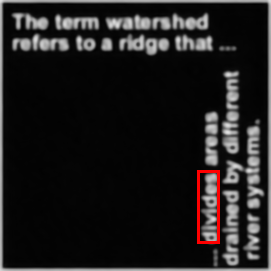}
        \put(0,0){%
            \fcolorbox{red}{white}{\includegraphics[width=0.03\textwidth]{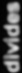}}%
        }
    \end{overpic}
} \hspace{0.1cm}
\subfloat[]{%
    \begin{overpic}[width=0.16\textwidth]{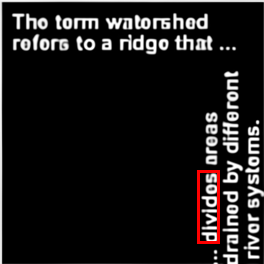}
        \put(0,0){%
            \fcolorbox{red}{white}{\includegraphics[width=0.03\textwidth]{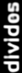}}%
        }
    \end{overpic}
}
\hspace{0.1cm}
\subfloat[]{%
    \begin{overpic}[width=0.16\textwidth]{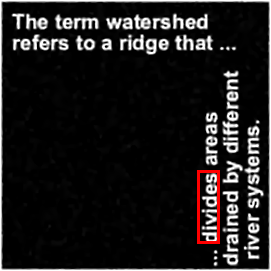}
        \put(0,0){%
            \fcolorbox{red}{white}{\includegraphics[width=0.03\textwidth]{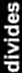}}%
        }
    \end{overpic}
}
\\[-0.3cm]
\caption{Restored Text. (a)DRUNet, (b)FFDNet, (c)DnCNN, (d)DPIR, (f)Proposed method.}
\label{Text_net}
\end{figure} 

For the Text image, our proposed method obtains relatively higher PSNR and FOM values. Although the metrics are slightly lower than those of the DPIR method, the visual comparison indicates that DPIR introduces noticeable distortions. For instance, the word ``divides'' is incorrectly restored as ``dlvldos''. In contrast, our method successfully preserves the structural properties of the word ``divides'', demonstrating its superior accuracy.

In summary, the qualitative and quantitative analyses substantiate the effectiveness of our proposed method in image restoration tasks. Compared with classical TV-based methods as well as representative learning-based models, our method  achieves superior performance in both visual quality and quantitative metrics, thereby affirming its practical utility and competitiveness among current state-of-the-art techniques.

\section{Conclusion}
\label{sec5}

In this paper, we proposed a novel semi-convergent stage-wise framework for adaptive total variation regularization, which effectively leverages the complementary strengths of both first- and higher-order TV regularizers. The convergence of the stage-wise alternating method was established both theoretically and numerically. Extensive experiments demonstrate that the proposed method  outperforms classical TV-based approaches and  several learning-based methods in terms of PSNR, SSIM, and FOM metrics. Both visual and quantitative results confirm the superiority of the proposed method in preserving fine details, yielding restored images that are visually pleasing and quantitatively robust. Nevertheless, some limitations remain, particularly under severe noise conditions, where the effectiveness of the proposed method is somewhat diminished. This suggests potential room for improvement in the design and integration of TV-based models. Future work will explore the extension of the hybrid alternating optimization framework to other advanced regularization models, aiming to further enhance restoration quality and computational efficiency. The promising results obtained in this study lay a solid foundation for such developments.


\bibliographystyle{siamplain}%
\bibliography{references}
\end{document}